\documentclass[11pt]{amsart}

\usepackage{amsmath,amssymb,amsthm,mathrsfs,enumerate}
\usepackage{comment,hyperref,tikz}
\hypersetup{
    colorlinks = true,
    citecolor = magenta,
    }

\theoremstyle{plain}
\newtheorem{theorem}{Theorem}[section]
\newtheorem{thm}[theorem]{Theorem}

\newtheorem{lemma}[theorem]{Lemma}
\newtheorem{proposition}[theorem]{Proposition}
\newtheorem{corollary}[theorem]{Corollary}

\theoremstyle{definition}
\newtheorem{definition}[theorem]{Definition}
\newtheorem{remark}[theorem]{Remark}

\theoremstyle{remark}
\newtheorem{example}{Example}

\numberwithin{equation}{section}

\newcommand{\mat}[1]{\left[ \begin{matrix} #1 \end{matrix} \right]}
\newcommand{\smat}[1]{\left[ \begin{smallmatrix} #1 \end{smallmatrix} \right]}

\newcommand{\ratk}{F}
\newcommand{\resk}{\mathfrak{f}}
\newcommand{\RR}{\mathcal{R}}
\newcommand{\PP}{\mathcal{P}}
\newcommand{\ep}{\varepsilon}
\newcommand{\Q}{\mathbb{Q}}
\newcommand{\p}{\varpi}
\newcommand{\val}{\mathrm{val}}
\newcommand{\extk}{E}

\newcommand{\bG}{\mathbf{G}}

\newcommand{\sG}{\mathsf{G}}
\newcommand{\sT}{\mathsf{T}}
\newcommand{\sU}{\mathsf{U}}

\newcommand{\Sp}{\mathrm{Sp}}
\newcommand{\SL}{\mathrm{SL}}

\newcommand{\GL}{\mathrm{GL}}

\newcommand{\St}{\mathrm{St}}
\newcommand{\sltwo}{\mathfrak{sl}_2(F)}

\newcommand{\g}{\mathfrak{g}}
\newcommand{\LieT}{\mathfrak{t}}

\newcommand{\LieS}{\mathfrak{s}}
\newcommand{\LieU}{\mathfrak{u}}

\newcommand{\buil}{\mathcal{B}}
\newcommand{\apart}{\mathcal{A}}

\newcommand{\Res}{\mathrm{Res}}
\newcommand{\Ind}{\mathrm{Ind}}
\newcommand{\cind}{\textrm{c-}\mathrm{Ind}}

\newcommand{\tr}{\mathrm{tr}}

\newcommand{\Cent}{\mathrm{Cent}}
\newcommand{\Norm}{\mathrm{Norm}}
\newcommand{\Lie}{\mathrm{Lie}}

\newcommand{\setorbit}{\mathscr{O}(0)}
\newcommand{\setorbitGamma}{\mathscr{O}(\Gamma)}
\newcommand{\nilp}{\mathcal{N}} 
\newcommand{\nil}{\mathrm{Nil}}
\newcommand{\orb}{\mathcal{O}}
\newcommand{\cO}{\mathcal{O}} 
\newcommand{\cone}{\mathop{Cone}} 
\newcommand{\WF}{\mathcal{WF}} 
\newcommand{\algy}{\dot{X}}
\newcommand{\dualy}{X}

\newcommand{\Sh}{\mathcal{S}}
\newcommand{\allbutone}{[G_x\backslash G/\Cent(\Gamma)]^{\text{deg}}}

\newcommand{\reg}{\mathrm{reg}}
\newcommand{\Ad}{\mathop{Ad}}
\newcommand{\real}{\mathbb{R}}

\newcommand{\sgn}{\mathsf{sgn}}
\newcommand{\diag}{\textrm{diag}}
\newcommand{\triv}{{\bf 1}}
\newcommand{\C}{\mathbb{C}}
\newcommand{\Z}{\mathbb{Z}}

\begin{document}

\title[LCE as branching rules]{The local character expansion as branching rules: nilpotent cones and the case of $\SL(2)$}
\author{Monica Nevins}
\email{mnevins@uottawa.ca}
\address{Department of Mathematics and Statistics, University of Ottawa, Canada K1N 6N5}
\thanks{Supported by  NSERC Discovery grant RGPIN-2020-05020.}
\date{\today}

\begin{abstract}
We show  there exist representations of each maximal compact subgroup $K$ of the $p$-adic group $G=\SL(2,\ratk)$, attached to each nilpotent coadjoint orbit, such that every irreducible representation of $G$, upon restriction to a suitable subgroup of $K$, is a sum of these five representations in the Grothendieck group.  This is a representation-theoretic analogue of the analytic local character expansion due to Harish-Chandra and Howe.  Moreover, we show for general connected reductive groups that the wave front set of many irreducible positive-depth representations of $G$ are completely determined by the \emph{nilpotent support} of their unrefined minimal $K$-types.
\end{abstract}
\maketitle

\section{Introduction}

The distribution character of an admissible representation of a $p$-adic group can be expressed, in a neighbourhood of the identity, as a linear combination of Fourier transforms of the finitely many nilpotent orbital integrals in the dual of the Lie algebra.  This remarkable theorem, known as the Harish-Chandra--Howe local character expansion, has many variations (such as expansions on neighbourhoods of other semisimple elements, or expansions in terms of other collections of orbital integrals \cite{KimMurnaghan2003,KimMurnaghan2006, Spice2018}) and many applications (such as determining the Gel'fand--Kirillov dimension of a representation, or relating to conjectural classifications such as the orbit method, or the local Langlands correspondence \cite{BarbaschMoy1997, CiubotaruMason-BrownOkada2021, CiubotaruMason-BrownOkada2022, JiangLiuZhang2022}).  Though it is primarily considered in characteristic zero, it also holds when the characteristic is sufficiently large and a suitable substitute for the exponential map exists \cite{CluckersGordonHalupczok2014}.

In this paper, we interpret the local character expansion as a statement in the Grothendieck group of representations of a maximal compact open subgroup, upon restriction to a subgroup of suitable depth,  for the case that $G=SL(2,F)$, where $F$ is a local nonarchimedean field of residual characteristic at least $3$.  In particular, we construct for each nilpotent orbit $\orb$ of $G$ in the dual of its Lie algebra $\g^*$ a (highly reducible) representation $\tau_x(\orb)$ of each maximal compact open subgroup $G_x$ with the following property.

\begin{theorem}\label{maintheorem}
Let $\pi$ be an irreducible admissible representation of $G=\SL(2,F)$ of depth $r\geq 0$ and let $x$ be a vertex in the building of $G$.  Then there exist integers $c_{x,\orb}(\pi)$ such that in the Grothendieck group of representations we have
\begin{equation}\label{LCEbrules}
\Res^G_{G_{x,r+}}\pi = \sum_{\orb} c_{x,\orb}(\pi) \Res^{G_x}_{G_{x,r+}}\tau_{x}(\orb)
\end{equation}
where $G_{x,r+}$ is the Moy--Prasad filtration subgroup of $G_x$ of depth $r+$, and the sum is over all nilpotent orbits in $\g^*$.
\end{theorem}

Moreover, the coefficients corresponding to the regular nilpotent orbits in this expansion are nonnegative integers and agree with those of the Harish-Chandra--Howe local character expansion (subject to suitable normalizations).   Note that while inherently expressing the same local nature of representations,  our statement holds with fewer restrictions on $\ratk$ than does the local character expansion, because it does not depend on the existence of a $G$-equivariant map, such as the exponential or a Cayley transform, from the Lie algebra to the group.

If $G$ is an inner form of $\mathrm{GL}_n(F)$, then an explicit decomposition of the form \eqref{LCEbrules} has been proven by Henniart and Vign\'eras in \cite{HenniartVigneras2023}; moreover, their local expansion holds for representations of $G$ over any field $R$ of characteristic not $p$.  They obtain the representations we denote here by $\Res_{G_{x,r+}}^{G_{x}}\tau_x(\orb)$ as $\Res_{G_{x,r+}}^G \Ind_P^G \triv$, for a suitable parabolic subgroup attached to $\orb$, vastly generalizing a result of Howe  \cite{Howe1974}.  Though such an elegant description is not directly available here (see also Remark~\ref{restrep}), our work does answer \cite[Questions 1.1, 1.2]{HenniartVigneras2023} for complex representations of $\SL(2,\ratk)$.

Now suppose $G$ is a general connected reductive group. In Section~\ref{S:Nilpotent}, we develop some theory towards establishing the direct relationship from the LCE to a decomposition like \eqref{LCEbrules}, as follows. 

The set of maximal orbits appearing in the local character expansion for an admissible representation $\pi$ is denoted $\WF(\pi)$; the closure of the union of these orbits is the wave front set of $\pi$.  For depth-zero representations $\pi$,  Barbasch and Moy \cite{BarbaschMoy1997} proved that $\WF(\pi)$  is determined by the depth-zero components of the restriction of $\pi$ to various maximal compact subgroups, through the theory of Gel'fand--Graev representations.

For a positive-depth representation with minimal $K$-type $\Gamma$ (in the sense of Moy and Prasad \cite{MoyPrasad1994}), we should instead infer  $\WF(\pi)$  from the   \emph{nilpotent support} $\nil(\Gamma)$ (Definition~\ref{D:nilpotentsupport}) of $\Gamma$.  This definition, of independent interest, depends strongly on the classification of nilpotent orbits using Bruhat--Tits theory \cite{BarbaschMoy1997,DeBackerNilpotent2002}.  In fact, in Proposition~\ref{P:cone=nil} we show that the algebraic notion of nilpotent support can be characterized as the set of nonzero nilpotent orbits appearing in the asymptotic cone on $\Gamma$, as defined in \cite{AdamsVogan2021}.  
In Theorem~\ref{T:nil=wf} (proof due to Fiona Murnaghan), we prove that  $\WF(\pi)$ is the set of maximal orbits of $\nil(\Gamma)$ whenever the $\Gamma$-asymptotic expansion \cite{KimMurnaghan2003} reduces to a single term.  

 This last result is similar to recent work of Ciubotaru and Okada, who show  that the depth-$r$ components of the restriction to certain compact open subgroups determine the wave front set of $\pi$ \cite{CiubotaruOkada2023}.  The idea of the nilpotent support is also central to \cite{CiubotaruOkada2023}, where they develop it using, among other things, the geometry of the associated finite reductive group.

Now again suppose that $G=\SL(2,\ratk)$.
Our result gives a second characterization of $\WF(\pi)$: it can be entirely determined from the \emph{non-typical} representations occurring in the restriction of $\pi$ to a maximal compact open subgroup, for $\pi$ of any depth.  That is, the asymptotic decomposition of $\Res_{G_x}\pi$ unfolds exactly as  the representations $\tau_x(\orb)$ for $\orb\in \WF(\pi)$.

For the case of a positive-depth representation $\pi$, our main theorem is stated in Theorem~\ref{T:posdepth2}, with the explicit values of the constant coefficient given in Proposition~\ref{P:npi}.  To prove the theorem, we first show that the restriction of $\pi$ to a maximal compact subgroup can be expressed entirely in terms of twists of the pair $(\Gamma,\chi)$ used in the construction of $\pi$ (Theorem~\ref{T:posdepth}), using results from \cite{Nevins2005,Nevins2013}.  Here, $\chi$ is a  character  of a torus $T=\Cent_G(\Gamma)$ that is realized by $\Gamma \in \g^*$, and the realization of the irreducible components of the restriction is framed in terms of a generalization (Proposition~\ref{P:Gxreps}) of a construction due to Shalika in his thesis.  
From this characterization, and a key technical result (Lemma~\ref{L:Shalikaequiv}), it follows that the expansion \eqref{LCEbrules} exists and has leading terms corresponding to the nilpotent support of $\Gamma$.
Since $\Gamma$ represents a minimal $K$-type of $\pi$ in the sense of Moy and Prasad \cite{MoyPrasad1994}, we independently recover from Theorem~\ref{T:nil=wf} that the maximal orbits in $\nil(\Gamma)$ coincide with $\WF(\pi)$.

For representations of depth zero, the principal technical difficulties lie in matching the depth-zero components with nilpotent orbits, particularly in the case of the twelve ``exceptional'' representations: the reducible principal series, the principal series composed of the trivial and the Steinberg representation, and the four special supercuspidal representations.  Once these are addressed, Theorem~\ref{T:zerodepth2} follows by carefully extracting the necessary branching rules from \cite{Nevins2005,Nevins2013}.  Again, the orbits in $\WF(\pi)$ are obtained both from the depth-zero components (via \cite{BarbaschMoy1997}) and from the asymptotic development of the branching rules.

At two crucial junctures we use information that is currently only known for $G=\SL(2,\ratk)$ and a handful of other small rank groups:  one is the explicit calculation of the asymptotic cone on any semisimple element of $\g^*$ (Section~\ref{S:sl2nil}); the other is the full knowledge of the representation theory of the maximal compact subgroups of $G$ (Section~\ref{S:Representations}).  
While the former seems a tractable and interesting question in general, the latter is quite daunting: it is not expected that we will achieve a classification of the representations of maximal compact open subgroups of $p$-adic reductive groups.  Note that a full classification is not necessary to prove the theorem: what is needed is a construction of an appropriate representation of $G_x$ attached to each nilpotent orbit, and we explore how this might be done in Section~\ref{S:gx}.


There are many interesting applications and open directions left to pursue. 
Evidently the overarching goal is to establish a result like \eqref{LCEbrules} for a large class of groups, using the tools presented here, or those developed in \cite{HenniartVigneras2023}.   To extend the work here, it may be fruitful to build representations of the groups $G_{x,0+}$ directly, rather than to construct representations of $G_{x,0}$; this has the advantage of avoiding the difficulties inherent at depth zero.  It may also allow for a more uniform treatment of all points $x$ of the building; in this paper, we consider only vertices, and the union of all $G_{x,r+}$ as $x$ runs over vertices is not equal to $G_{r+}$ in general.  

In another direction: the $\Gamma$-asymptotic expansions of  \cite{KimMurnaghan2003,KimMurnaghan2006} describe the character of a positive-depth representation in a larger neighbourhood than does the local character expansion, by incorporating a minimal $K$-type $\Gamma$.  Then Theorem~\ref{T:posdepth} can be interpreted as analogously formulating these expansions in terms of branching rules.  It would be interesting to explore this idea further.

The paper is organized as follows.  We set our notation in Section~\ref{S:Notation} and then present some background on the local character expansion that provide the motivation and context for our results.  In Section~\ref{S:Nilpotent} we consider a general connected reductive group $G$.  We define the nilpotent support of an element $\Gamma$ of $\g^*$, show it defines the asymptotic cone of $\Gamma$, and relate this to the wave front set via the theory of $\Gamma$-asymptotic expansions. 

We then specialize to $G=\SL(2,\ratk)$.
In Section~\ref{S:sl2nil} we characterize the nilpotent cones $\nil(\Gamma)$ in many ways (Proposition~\ref{P:asymptotic_orbit_invariance}) and compute them explicitly.     In Section~\ref{S:Representations} we recall the construction of certain irreducible representations of $\SL(2,\RR)$ by Shalika in his 1966 thesis \cite{Shalika1966}, and then rephrase it using Bruhat--Tits theory and derive some consequences.  This allows us to define, for each vertex $x\in\buil(G)$, each nilpotent orbit $\cO \subset \g^*$, and each central character $\zeta$ a representation $\tau_x(\cO,\zeta)$ of $G_x$.   

We prove our main theorems for representations of positive depth in Section~\ref{S:posdepth} and for representations of depth zero in Section~\ref{S:depthzero}.  
 We conclude with two brief applications of Theorem~\ref{maintheorem} in Section~\ref{S:applications}: an explicit formula for the functions $\hat{\mu}_\orb$ in terms of the trace character of the representation $\tau_x(\orb)$ of the compact group $G_x$; and an explicit polynomial expression for $\dim(\pi^{G_{x,2n}})$ (in the spirit of \cite{HenniartVigneras2023}) whose existence is predicted by the local character expansion. 
 
\subsection*{Acknowledgements}  This work was instigated by a question posed to the author by David Vogan and has benefitted enormously from many conversations with him in the  online research community on Representation Theory and Noncommutative Geometry sponsored by the American Institute of Mathematics.   
The approach to $\nil(\Gamma)$ given here was signficantly refined through conversations with Fiona Murnaghan and  Loren Spice.  
This work progressed over a period of visits to 
many colleagues, and benefitted from their comments and interest:  Vincent S\'echerre,  Laboratoire de Mathématiques de Versailles, Universit\'e Paris-Saclay; Anne-Marie Aubert, Institut de Mathématiques de Jussieu-Paris Rive Gauche, Universit\'e de Paris/Sorbonne Universit\'e and Jessica Fintzen, Universit\"at Bonn.  It is a true pleasure to thank all of these generous people.

\section{Notation and background}\label{S:Notation}

Let $\ratk$ be a local nonarchimedean field of residual characteristic $p\neq 2$, with integer ring $\RR$, maximal ideal $\PP$ and residue field $\resk$ of cardinality $q$.  We impose additional hypotheses on $p$ in Section~\ref{SS:prestrict}, below.   
Fix once and for all an additive character $\psi$ of $\ratk$ that is trivial on $\PP$ and nontrivial on $\RR$.  Fix a uniformizer $\p$ and normalize the valuation on $\ratk$ (and any extension thereof) by $\val(\p)=1$.  We write $\val(0):=\infty$.

Let $\bG$ denote a connected reductive algebraic group defined over $\ratk$ whose group of $\ratk$-rational points is denoted $G$; we use $\g=\Lie(\bG)(\ratk)$ to denote its Lie algebra over $\ratk$.  We simplify notation by refering to tori, Borel subgroups and parabolic subgroups of $G$ when we mean the $F$-points of such algebraic $F$-subgroups of $\bG$, and denote them in roman font.  Let $G^{\reg}$, respectively $\g^{\reg}$, denote the set of regular elements of $G$, respectively $\g$.
The group $G$ acts on $\g$ via the adjoint action $\Ad$ and on its dual $\g^*$ via the coadjoint action $\Ad^*$; we abbreviate these by both $g\cdot X$ or ${}^gX$ for $g\in G$ and $X$ in $\g$ or $\g^*$.  Similarly, if $H$ is a subgroup of $G$ we write ${}^gH$ for the group $gHg^{-1}$.

An element $X\in\g^*$ or $\g$ is called \emph{semisimple} (or \emph{almost stable}) if its $G$-orbit is closed.  We define $X\in \g^*$ or $\g$ to be \emph{nilpotent} if there exists an $F$-rational one-parameter subgroup $\lambda \in X_*(\bG)$ such that $\lim_{t\to 0}{}^{\lambda(t)}X = 0$.  By \cite[\S2.5]{AdlerDeBacker2002}, this is equivalent to a more usual definition that the closure of the coadjoint orbit in the rational topology contains $0$.    
We say the one-parameter subgroup $\lambda$ is \emph{adapted} to $X$ \cite[Definition 4.5.6]{DeBackerNilpotent2002}, if ${}^{\lambda(t)}X =t^2X$.  
 We write $\nilp^*$ for the set of nilpotent elements of $\g^*$ and $\setorbit$ for the (finite) set of $G$-orbits in $\nilp^*$. 
  
We sometimes specify a group of matrices merely by the sets in which its entries lie; in this case, that the resulting subgroup is the intersection of this set with $G$ is understood.  We write $\lceil t \rceil = \min\{n\in \Z\mid n\geq t\}$ and $\lfloor t \rfloor = \max\{n\in \Z \mid n\leq t\}$.  Write $\Cent_G(S)$ for the centralizer in $G$ of the element or set $S$.  We may write $[\sigma]$ for the trace character of a representation $\sigma$ of a finite or compact group.  The trivial representation is denoted $\triv$, and the characteristic function of a subset $S$ is denoted $\triv_S$.

\subsection{The Bruhat--Tits building and Moy--Prasad filtration subgroups}
Let $\buil(G)=\buil(\bG,\ratk)$ denote the (enlarged) Bruhat-Tits building of $G$; then to each $x\in \buil(G)$ we associate its stabilizer $G_x$, which is a compact subgroup of $G$ containing the parahoric subgroup $G_{x,0}$.  
These admit a Moy-Prasad filtration by normal subgroups $G_{x,r}$ with $r\in \mathbb{R}_{\geq 0}$ defined relative to the valuation on $\ratk$.  We briefly recap the definition; for a careful and detailed summary, see for example \cite[\S2]{FintzenMP2021}.  

To define $G_{x,r}$, choose an apartment $\apart \subset \buil(G)$ containing $x$; this is the affine space over $X_*(T)\otimes_\Z\real$ for some maximal split torus $T$ of $G$ and we write $\apart = \apart(G,T)$.  Let $\Phi=\Phi(G,T)$ denote the corresponding root system and $\Psi$ the set of affine roots, viewed as functions on $\apart$.  For each root $\alpha\in \Phi$, let $U_\alpha$ denote the corresponding root subgroup.  The affine roots $\psi$ with gradient $\alpha$ define a filtration of $U_\alpha$ by compact open subgroups $U_\psi$.   Let $C=\Cent_G(T)$; it contains a parahoric subgroup $C_0$, and a filtration by compact open normal subgroups $C_r$, $r>0$, that is independent of the point $x\in \apart$.  Then for any $r\geq 0$ we define compact open subgroups
$$
G_{x,r} = \langle C_r, U_\psi \mid \psi\in \Psi, \psi(x)\geq r\rangle;
$$
if $r=0$ this is the parahoric subgroup and for $r>0$ it is a Moy--Prasad filtration subgroup of $G_{x,0}$.  It is independent of the choice of apartment containing $x$. The Moy--Prasad filtration is $G$-equivariant; for example ${}^gG_{x,r}=G_{gx,r}$ for all $x\in \buil(G)$ and $r\geq 0$.

Similarly, the Lie algebra $\g$ admits a filtration $\g_{x,r}$ by $\RR$-modules indexed by $r\in \mathbb{R}$, as follows.  Let $\mathfrak{t}$ denote the Lie algebra of $T$, $\mathfrak{c}$ its centralizer in $\g$ and for each $\alpha \in \Phi$, let $\g_\alpha$ denote the corresponding root subspace.  These subspaces admit filtrations by $\RR$-submodules $\mathfrak{c}_r$ with $r\in \real$ and $\g_{\psi}$ for $\psi\in \Psi$, respectively, such that
\begin{equation}\label{E:MPfiltration}
\g_{x,r} = \mathfrak{c}_r \oplus \bigoplus_{\alpha} \g_{\alpha,x,r},
\end{equation}
where $\g_{\alpha,x,r}$ is the union of the $\RR$-submodules $\g_\psi$ such that $\psi\in \Psi$, the gradient of $\psi$ is $\alpha$ and $\psi(x)\geq r$.
We write 
$$
G_{x,r+}=\bigcup_{s>r}G_{x,s}, 
\quad \text{and} \quad 
\g_{x,r+} = \bigcup_{s>r} \g_{x,s}.
$$  
By \cite[\S1.6]{Adler1998}, there exists a mock exponential map $e=e_x\colon \g_{x,0+}\to G_{x,0+}$ that induces the Moy--Prasad isomorphism $\g_{x,r}/\g_{x,2r} \cong G_{x,r}/G_{x,2r}$ for any $r>0$ (among other desirable properties). 
Writing $\langle X,Y\rangle$ for the natural pairing of $X\in \g^*$ with $Y\in \g$, the Moy-Prasad filtration on the dual of the Lie algebra is defined by $\g^*_{x,r} = \{X\in \g^*\mid \forall Y\in \g_{x,(-r)+}, \langle X,Y\rangle\in \PP \}$. We again define $\g^*_{x,r+} = \cup_{s>r}\g^*_{x,s}$.

Finally, for any $r\geq 0$ we define $G$-stable subsets
$$
G_{r} = \bigcup_{x\in \buil(G)}G_{x,r} \quad \text{and} \quad G_{r+} = \bigcup_{x\in \buil(G)}G_{x,r+}.
$$
For any real number $r$ we do the same to define $\g_r$ and $\g_{r+}$.

If $(\pi,V)$ is an irreducible admissible representation of $G$, then its depth is defined as the least real number $r\geq 0$ such that there exists $x\in\buil(G)$ for which $V^{G_{x,r+}} \neq \{0\}$.  
We define the depth of a smooth irreducible representation $\rho$ of $G_x$, for fixed $x$, in the same way; this is equivalent to the least $r\geq 0$ for which $\rho$ factors through $G_x/G_{x,r+}$.

\subsection{Restrictions on \texorpdfstring{$p$}{p}}\label{SS:prestrict}

We impose the restriction that $G$ splits over a tamely ramified extension of $\ratk$ and that $p$ does not divide the order of the absolute Weyl group of $\bG$. One of the main results of \cite{Fintzen2021} is that  this is sufficient to ensure that all irreducible admissible representations are tame.  Combining \cite[Lemma 2.2, Table 1]{Fintzen2021} and \cite[\S 1]{AdlerRoche2000}, one sees that the hypotheses of  \cite[Hypothesis 2.1.1]{Adler1998} or \cite[Prop 4.1]{AdlerRoche2000} hold, so that there is a non-degenerate $G$-invariant bilinear form on $\g$ under which $\g^*_{x,r}$ and $\g_{x,r}$ are identified for all $x$ and $r$.  For $\mathbf{G}=\SL(2)$ and $p\neq 2$ we may take the trace form, and define for each $\dot{X}\in \g$ the element $X\in\g^*$ by $\langle X,\cdot\rangle =\tr(\dot{X}\cdot)$. 

We also impose the hypotheses of \cite[\S4]{DeBackerNilpotent2002} to obtain the classification of nilpotent orbits; this requires the use of $\sltwo$ triples over the residue field as well as some properties of a mock exponential map.   By recent work of Stewart and Thomas, \cite{StewartThomas2017} the former condition is satisfied for $p>h$, where $h$ is the Coxeter number of $G$.  To satisfy all hypotheses for $G=\SL(2,F)$, it suffices to take $p\geq 3$.

In contrast, to state the local character expansion, which relates a function on the group to one on the Lie algebra, one needs a $G$-equivariant map $ \g_{0+}\to G_{0+}$ satisfying \cite[Hypothesis 3.2.1]{DeBackerHomogeneity2002}.  Such a map, which we'll simply denote $\exp$, can exist in large positive characteristic (see, for example, the discussion in \cite[\S2]{CluckersGordonHalupczok2014}); in characteristic zero, \cite[Lemma B.0.3]{DeBackerReeder2009} gives an effective lower bound on $p$.  
For $G=\SL(2,F)$, this entails in characteristic zero that $p>e+1$ where $e$ is the ramification index of $F$ over $\Q_p$, for example.

\subsection{The local character expansion}

As detailed in the expanded notes \cite{HarishChandra1999}, Harish-Chandra proved in the 1970s that the distribution character of an irreducible admissible representation $\pi$ of $G$, which is given on $f\in C_c^\infty(G)$ by
$$
\Theta_\pi(f)  = \tr\int f(g)\pi(g) \, dg,
$$
is well-defined and representable by a function, which we also denote $\Theta_\pi$, that is locally integrable on $G$ and locally constant on the set $G^\reg$ of regular semisimple elements of $G$ (see \cite[\S13]{AdlerKorman2007} and the discussion therein).  


Similarly, to each coadjoint orbit $\orb\subset \g^*$ we associate its  orbital integral, given on $f\in C_c^\infty(\g^*)$ by
\begin{equation}\label{E:orbitalintegral}
\mu_\orb(f) = \int_{\orb} f(X) \, d\mu_\orb(X)
\end{equation}
where $d\mu_\orb$ is a Radon measure \cite{RangaRao1972}.  
Relative to $\psi$, the fixed additive character of $\ratk$, the Fourier transform of $f\in C_c^\infty(\g)$ is a function $\hat{f} \in C_c^\infty(\g^*)$. The Fourier transform of the orbital integral $\mu_\orb$ is the distribution given on $f\in C_c^\infty(\g)$ by $\widehat{\mu_\orb}(f) = \mu_\orb(\hat{f})$.  Then $\widehat{\mu_\orb}$ is representable by a locally integrable function on $\g$ that is locally constant on $\g^{\reg}$ \cite[Theorem 4.4]{HarishChandra1999}.  
We set $\g^{\reg}_{r+}:=\cup_{x\in \buil(G)}\g_{x,r+}\cap \g^{\reg}$.

The local character expansion expresses that these finitely many functions $\widehat{\mu_\orb}$, for $\orb\in \setorbit$, form a basis, in a neighbourhood of $0$,  for the space of locally integrable $G$-invariant functions that are locally constant on $\g^{\reg}$.  The nature of the expansion was  first proven for $G=\GL(n,F)$ in characteristic $0$ by Roger Howe \cite{Howe1974} and then in the generality of connected reductive groups in characteristic zero by Harish-Chandra \cite{HarishChandra1999}.  Cluckers, Gordon and Halupczok proved its validity in large positive characteristic in \cite{CluckersGordonHalupczok2014}; Adler and Korman proved an analogous result for expansions centered at other semisimple elements in \cite{AdlerKorman2007}.

The precise domain on which the local character expansion holds was conjectured by Hales, Moy and Prasad \cite{MoyPrasad1994} and proven in \cite{Waldspurger1995} for a large class of groups and by \cite{DeBackerHomogeneity2002} in the following generality.  

 \begin{thm}[The Local Character Expansion]
If $\pi$ is an irreducible admissible representation of $G$ of depth $r$, then there exist unique $c_\orb(\pi)\in \C$ such that for all $X\in \g^{\reg}_{r+}$ we have
\begin{equation}\label{E:LCE}
\Theta_\pi(\exp(X)) = \sum_{\orb \in \setorbit} c_\orb(\pi) \widehat{\mu_\orb}(X).
\end{equation}
\end{thm}

We denote by $\WF(\pi)$ the set of \emph{maximal} nilpotent orbits $\orb$ such that $c_\orb(\pi)\neq 0$, where the ordering is taken in the local topology; this is the set denoted $\mathrm{WF}^{\mathrm{rat}}(\pi)$ in \cite{Tsai2023}.
In \cite{Heifetz1985}, Heifetz defined and developed the analytic notion of the \emph{wave front set} of a representation of a $p$-adic group, in analogy with the work of Howe \cite{Howe1979} in the real case.  In \cite{Przebinda1990}, Przebinda proved that the wave front set coincides with the support of the right side of \eqref{E:LCE}, which is the closure of the union of these orbits.  Recent work of Tsai \cite{Tsai2022,Tsai2023, Tsai2023a} has shown that the orbits of $\WF(\pi)$ may fail to be stably conjugate. 

Finally, note that for $G=\GL(n,F)$, Howe proved that for each $\orb \in \setorbit$,  there is a corresponding parabolic subgroup $P$ such that in a neighbourhood $O$ of $0\in \g$ we have
$$\widehat{\mu_\orb}|_O =\Theta_\pi\circ \exp|_O$$
where $\pi = \Ind_P^G \bf{1}$ \cite[Lemma 5]{Howe1974}.   This has recently been generalized to all inner forms of $G$, and representations over fields of characteristic not equal to $p$, in \cite[Theorem 1.3]{HenniartVigneras2023}. 
In the same vein, for $\SL(2,\ratk)$, the functions $\widehat{\mu_\orb}$ are almost equal to the characters of special unipotent representations (see \eqref{E:specialchar}).
We cannot expect such equalities in general as, for example, for classical  groups nonspecial orbits cannot occur in $\WF(\pi)$ for any $\pi$ \cite[Thm 1.4]{Moeglin1996}).  The main goal in this paper is to propose an example of a  weaker form of the Howe--Henniart--Vign\'eras theorem, based on representations of a maximal compact open subgroup, that one may hope can hold true in general.

\section{Nilpotent orbits and nilpotent support} \label{S:Nilpotent}

In this section, $G$ is an arbitrary connected reductive group, subject to the hypothesis on $p$ of Section~\ref{SS:prestrict}.  We define the (local) nilpotent support of an element of $\g^*$, and relate this both to the asymptotic cone and to the wave front set of a representation of positive depth.

\subsection{Degenerate cosets and nilpotent orbits}
In \cite[\S3]{AdlerDeBacker2002}, Adler and DeBacker generalize ideas of Moy and Prasad to establish, for connected reductive groups, that for all $r\in \mathbb{R}$ 
$$
\g^*_r = \bigcap_{x\in \buil(G)}(\g^*_{x,r}+\nilp^*),
$$
where $\g^*_r := \bigcup_{x\in \buil(G)} \g^*_{x,r}$. 
They further show that 
$$
\nilp^* = \bigcap_{r\in \mathbb{R}} \g_r^*.
$$  
Given $x\in \buil(G)$ and $X\in \g^*\setminus \{0\}$, the \emph{depth of $X$ at $x$} is the unique value $t = d_x(X)$ such that $X\in \g^*_{x,t}\setminus \g^*_{x,t+}$.   When $X$ is not nilpotent, they prove that the \emph{depth of $X$}, given by
$$
d(X)=\max\{d_x(X)\mid x\in \buil(G)\} = \max\{r\mid X\in \g^*_r\}
$$
is well-defined and rational. For $X$ nilpotent, we set $d(X)=\infty$.  Depth is $G$-invariant.  

For semisimple $\Gamma\in \g^*$, let $T\subset \Cent_G(\Gamma)$ be a maximal torus with associated absolute root system $\Phi(G,T)$.  Then $\Gamma$ is called \emph{good} if for all $\alpha \in \Phi(G,T)$, we have 
$\val(\Gamma(d\alpha^\vee(1)))\in \{d(\Gamma),\infty\}$.  
By \cite[Thm 2.3.1]{KimMurnaghan2003}, if $\Gamma$ is good then the set of points $x\in \buil(G)$ at which $d_x(\Gamma)$ attains its maximum value $d(\Gamma)$ is exactly $\buil(\Cent_G(\Gamma))\subset \buil(G)$.

For any $\Gamma \in \g^*$ set $d=d_x(\Gamma)$. The coset $\Gamma+\g^*_{x,d+}$ is called \emph{degenerate} if it contains a nilpotent element $X\in \nilp^*$.  From the relations above it follows that this  happens if and only if $d < d(\Gamma)$.
In \cite[\S5]{DeBackerNilpotent2002}, DeBacker proves that the set of nilpotent $G$-orbits meeting a degenerate coset $\Gamma+\g^*_{x,d+}$ has a unique minimal element with respect to the (rational) closure relation on orbits, which we'll denote $\mathcal{O}(\Gamma,x)$. This generalizes a result of Barbasch and Moy \cite[Prop 3.1.6]{BarbaschMoy1997} for $d=0$, which was integral to their determination of the wave front set of a depth zero representation.

To classify nilpotent orbits in this way, DeBacker proceeds as follows.  Identify $\g$ and $\g^*$.  Given a nilpotent element $X\in \g$, complete $X$ to an $\sltwo$ triple $(X,H,Y)$.  Choose $r\in \mathbb{R}$ and create the building set 
$$
\buil_r(X,H,Y) = \{x\in \buil(G) \mid X \in \g_{x,r}, \ H \in \g_{x,0}, \ Y\in \g_{x,-r}\};
$$
he proves this set is a nonempty, closed, convex subset of $\buil(G)$ with the property that for all $x\in \buil_r(X,H,Y)$ we have $\cO(X,x)=G\cdot X$.  

\begin{remark}\label{rem:coverage}
Note that for each $g\in G$, we have $\buil_r({}^gX,{}^gH,{}^gY)={}^g\buil_r(X,H,Y)$, and for fixed $X$ the union of these need not cover $\buil(G)$.  Moreover, if $\mu$ is a one-parameter subgroup adapted to this triple, then by \cite[Remark 5.1.5]{DeBackerNilpotent2002}, $\buil_r(X,H,Y)=\buil_0(X,H,Y)+\frac{r}{2}\mu$, where this sum is taken in any apartment in $\buil(C_G(\mu))$.  It follows that (if the rank of $G$ is greater than $1$) there exist orbits $\cO$ (such as ones for which $\buil_r(X,H,Y)$ is a point) for which there exist $y\in \buil(G)$ such that $\cO \neq \cO(X,y)$ for any $X\in \cO$. For example, in $\Sp(4,F)$, the principal nilpotent orbits are only obtained along certain lines emanating from vertices.  
\end{remark}

\subsection{Nilpotent support and nilpotent cones}

We now explore different ways to understand the asymptotic nilpotent support of a general element $\Gamma\in \g^*$ and show their equivalence.

\begin{definition}\label{D:nilpotentsupport}
Let $\Gamma\in\g^*$.  If $x\in \buil(G)$, then the \emph{local nilpotent support at $x$} of $\Gamma$ is  
$$
\nil_x(\Gamma) = \{ \cO({}^g\Gamma,x) \mid g\in G, d_x(g\cdot \Gamma)<d(\Gamma)\},
$$
which is the set of nilpotent orbits defined by degenerate cosets at $x$ of elements of the $G$-orbit of $\Gamma$.  On the other hand, the \emph{nilpotent support} of $\Gamma$ is 
$$
\nil(\Gamma) = \{ \cO(\Gamma,x) \mid x\in \buil(G), d_x(\Gamma)<d(\Gamma)\},
$$
the set of nilpotent orbits corresponding to any (nontrivial) degenerate coset of $\Gamma$.
\end{definition}

Note that if $\Gamma$ is nilpotent, then $\nil(\Gamma) = G\cdot \Gamma$. More generally, for any $g\in G$, we have $d_x(\Gamma)=d_{gx}({}^g\Gamma)$ and ${}^g(\Gamma+\g^*_{x,d+})= {}^g \Gamma + \g^*_{gx,d+}$.  
Thus $\cO(\Gamma,x)=\cO({}^g\Gamma,gx)$ and
$$
\nil(\Gamma)= \bigcup_{x\in \buil(G)}\nil_x(\Gamma),
$$
that is, the nilpotent support is the union of the local nilpotent supports, and  $\nil(\Gamma)$ is an invariant of the $G$-orbit of $\Gamma$.  One may alternately restrict this union to one over the points in a fundamental domain for the action of $G$ on $\buil(G)$.



By Remark~\ref{rem:coverage}, when the rank of $G$ is greater than $1$, not all nilpotent orbits will occur as some $\cO(\Gamma,x)$ for a given point $x\in \buil(G)$, so $\nil_x(\Gamma)\neq \nil(\Gamma)$ in general.
  Even when these sets are equal, as for $\SL(2,\ratk)$ (see Proposition~\ref{P:asymptotic_orbit_invariance}), they are interesting subsets of the nilpotent cone (see Lemma~\ref{L:NilGamma}).

On the other hand, the asymptotic cone on an element $\Gamma$ is defined in \cite[Def 3.9]{AdamsVogan2021} analytically as follows.

\begin{definition}
Let $\Gamma\in \g^*$. The \emph{asymptotic cone} on $\Gamma$ is the set
$$
\cone(\Gamma) = \{X\in \g^* \mid \exists \ep_i \to 0, \ep_i \in \ratk^\times, \exists g_i\in G, \lim_{i\to\infty}\ep_i^2 {\Ad}^*(g_i)\Gamma = X\}.
$$
\end{definition}

This is a closed, nonempty union of nilpotent orbits of $G$ on $\g^*$. 

\begin{proposition}\label{P:cone=nil}
Let $\Gamma\in \g^*$.  Then the nonzero $G$-orbits occurring in the asymptotic cone of $\Gamma$ are those  in its nilpotent support, that is, $$\cone(\Gamma) = \bigcup_{\cO \in \nil(\Gamma)} \cO \cup \{0\}.$$ 
\end{proposition}

\begin{proof}
It suffices to prove this result for $\Gamma \in\g$, where we may apply the theory of $\sltwo$ triples.

Let $\Gamma \in \g$ have depth $r\leq \infty$ and let $\mathcal{O}\in \nil(\Gamma)$.  Then there exists $x\in \buil(G)$ and $d<r$ such that $d_x(\Gamma)=d$ and $\mathcal{O}=\mathcal{O}(\Gamma,x)$.  Choose a representative 
$$
X\in \mathcal{O}(\Gamma,x) \cap (\Gamma+\g_{x,d+}).
$$
Choose an $\sltwo$ triple $(X,H,Y)$ and the corresponding one-parameter subgroup $\mu$ adapted to $X$.
By \cite[Lemma 5.2.1]{DeBackerNilpotent2002}, we have
$$
X + \g_{x,d+} = \Ad(G_{x,0+})(X+C_{\g_{x,d+}}(Y)).
$$
Therefore there exist $g \in G_{x,0+}$ and $C\in C_{\g_{x,d+}}(Y)$ for which
$$
\Gamma = \Ad(g^{-1})(X+C).
$$
Note that $C_{\g}(Y)$ is spanned by the lowest weight vectors of $\mathrm{ad}(H)$, so we may decompose $C=\sum_{i\leq 0}C_i$ where $\Ad(\mu(t))C_i = t^iC_i$ for all $t\in \ratk^\times$.  Similarly, for all $t\in \ratk^\times$ we have $\Ad(\mu(t))X = t^2X$.  
Therefore 
$$
\lim_{t\to 0} t^2\Ad(\mu(t^{-1})g)\Gamma =
\lim_{t\to 0} t^2\Ad(\mu(t^{-1}))(X+C)=
X
$$
so $X\in \cone(\Gamma)$.  Since $\cone(\Gamma)$ is $G$-invariant, we deduce $\mathcal{O}\subset \cone(\Gamma)$.

Conversely, let $X\in \cone(\Gamma)$ be nonzero, so that there exists a sequence of elements $\ep_i \in \ratk^\times$, with $\ep_i\to 0$, and a sequence of elements $g_i \in G$, such that
$$
\lim_{i\to \infty} \ep_i^2 \Ad(g_i)\Gamma = X. 
$$
Complete $X$ to an $\sltwo$ triple $(X,H,Y)$ and choose a point $x\in \buil_0(X,H,Y)$.   Since the given sequence converges to $X$, it enters the neighbourhood $X+\g_{x,0+}$ so we may choose $i\in \mathbb{N}$ such that
$$
\ep_i^2\Ad(g_i)\Gamma \in X + \g_{x,0+}.
$$
It follows that $\Ad(g_i)\Gamma \in \ep_i^{-2}X + \g_{x,-2\val(\ep_i)+}$, a nontrivial degenerate coset of depth $-2\val(\ep_i)$. Since $(\ep_i^{-2}X, H, \ep_i^2Y)$ is again an $\sltwo$ triple and $\buil_{-2\val(\ep_i)}(\ep_i^{-2}X, H, \ep_i^2Y) = \buil_0(X,H,Y)$, we infer that the minimal nilpotent orbit meeting this coset is $\Ad(G)(\ep_i^{-2}X) = \Ad(G)X$.    Thus $\Ad(G)X=\mathcal{O}({}^{g_i}\Gamma,x)\in \nil_x(\Gamma)\subset \nil(\Gamma)$, as required.
\end{proof}

\subsection{Connection with the wave front set of a positive depth representation}
Suppose now that $\pi$ is an irreducible admissible representation of $G$ of depth $r$ with good minimal $K$-type $\Gamma$ of depth $-r$  (in the sense of \cite[Def 2.4.3, 2.4.6]{KimMurnaghan2003}).  Then, under suitable hypotheses (that are satisfied if $\ratk$ has characteristic zero and the exponential map converges on $\g_{0+}$), Kim and Murnaghan prove a version of the local character expansion that is valid on the strictly larger neighbourhood $\g^{\reg}_r$.  The $\Gamma$-asymptotic expansion \cite[Thm 5.3]{KimMurnaghan2003} asserts that there exist complex coefficients $c_{\orb'}(\pi)$ such that for any $X\in \g^{\reg}_r$ we have 
\begin{equation}\label{E:KM}
\Theta_\pi(\exp(X)) = \sum_{\orb' \in \setorbitGamma} c_{\orb'}(\pi) \widehat{\mu_{\orb'}}(X),
\end{equation} 
where $\setorbitGamma$ denotes the set of $G$-orbits in $\g^*$ with $\Gamma$ in their closure, and for $\orb'\in \setorbitGamma$, $\widehat{\mu_{\orb'}}$ denotes the Fourier transform of the corresponding orbital integral \eqref{E:orbitalintegral}.

This yields a special case of interest: that of the expansion \eqref{E:KM} having a single nonzero term  $c_{\orb'}(\pi)\widehat{\mu_{\orb'}}$ corresponding to $\orb'=G\cdot \Gamma$.  We claim this happens, for example, when $G'=\Cent_G(\Gamma)$ is compact-mod-centre, such as when $\Gamma$ is a regular element.
 Namely, let $\g'$ denote the Lie algebra of $G'$. 
Then the set $\setorbitGamma$ indexing the sum in \eqref{E:KM} is in bijective correspondence with the set of nilpotent $G'$-orbits in $(\g')^*$, which is the singleton $\{G\cdot\Gamma\}$  under this hypothesis.

\begin{theorem}\label{T:nil=wf}
Let $\pi$ be an irreducible representation of $G$ of depth $r>0$, and let $\Gamma \in \g^*$ be a good minimal $K$-type of $\pi$ such that $\pi$ admits a $\Gamma$-asymptotic expansion.    Suppose further that this expansion has a unique nonzero term, corresponding to the Fourier transform of the orbital integral corresponding to $\Gamma$ itself.  
Then $\WF(\pi)$ coincides with the maximal elements of $\nil(\Gamma)$; that is, the asymptotic cone on $\Gamma$ is the wave front set of $\pi$.
\end{theorem}

The following proof was communicated to me by  Fiona Murnaghan. 

\begin{proof}
We are given that on $\g^{\reg}\cap\g_r$,   $\Theta_\pi\circ \exp = t\widehat{\mu_{G\cdot\Gamma}}$ for some nonzero scalar $t$. 
Applying the inverse Fourier transform to the local character expansion \eqref{E:LCE} of $\Theta_\pi$ (which is valid on  the smaller set $\g^{\reg}\cap\g_{r+}$)
we may write the equality of distributions  
\begin{equation}\label{E:germ}
t\mu_{G\cdot \Gamma}=\sum_{\cO \in \setorbit} c_{\cO}(\pi) \mu_\cO
\end{equation}
which by \cite[Cor 3.4.6]{DeBackerHomogeneity2002} holds in particular for all compactly supported functions on $\g^*/\g^*_{x,-r}$, for any $x \in \buil(G)$.  So let $x \in \buil(G)$ and let $d$ be such that $\g_{x,d+}^*\supset \g_{x,-r}^*$.  Given a nonzero coset $\xi\in \g^*_{x,d}/\g^*_{x,d+}$ let $\triv_\xi$ denote the characteristic function of this subset of $\g^*$.  Note that if $X\in \xi\cap \cO$ for some (not necessarily nilpotent) $G$-orbit $\cO$, then this intersection contains the open set $G_{x,0+}\cdot X$ as well.  Thus we have
\begin{equation}\label{*}
\mu_\cO(\triv_\xi) = 0 \iff \xi\cap\cO=\emptyset.
\end{equation}


Now suppose that $\cO \in \setorbit$, and choose $x\in \buil(G)$ and $\xi=X+\g^*_{x,d+}$ with $\g^*_{x,d+}\supset \g^*_{x,-r}$ with the property that $\cO = \cO(X,x)$.  The minimality of $\cO(X,x)$ proven by DeBacker implies that any nilpotent orbit $\cO'$ meeting $\xi$ (or equivalently, by \eqref{*}, satisfying $\mu_{\cO'}(\triv_\xi)\neq 0$) must contain $\cO$ in its closure.

Suppose first that $\cO$ is not in the wave front set $\cup_{\cO' \in \WF(\pi)}\overline{\cO'}$ of $\pi$.  Let $\cO'\in \setorbit$ be such that $c_{\cO'}(\pi)\neq 0$; then $\cO'$ is in the wave front set, so  $\cO\not\subset\overline{\cO'}$.  This implies by the preceding paragraph that 
$\mu_{\cO'}(\triv_\xi) = 0$.  As this holds for all such $\cO'$, we conclude from \eqref{E:germ} that $\mu_{G\cdot\Gamma}(\triv_\xi)=0$, whence  by \eqref{*} we have $\xi\cap G\cdot \Gamma = \emptyset$, and thus $\cO\notin \nil(\Gamma)$.  It follows that every $\cO\in \nil(\Gamma)$ lies in the wave front set of $\pi$.


Now suppose  $\cO \in \WF(\pi)$; that is,  it is maximal among nilpotent orbits with nonzero coefficient in \eqref{E:germ}.  Thus the preceding argument implies $\mu_{\cO'}(\triv_\xi)=0$ for all $\cO'\neq \cO$ in the wave front set.  Therefore \eqref{E:germ} yields $t\mu_{G\cdot \Gamma}(\triv_\xi)=c_{\cO}(\pi) \mu_{\cO}(\triv_\xi) \neq 0$, so by \eqref{*}, $\xi$ must meet $G\cdot \Gamma$ and thus $\cO \in \nil(\Gamma)$. 
Hence, the maximal elements of $\nil(\Gamma)$ coincide with $\WF(\pi)$.
\end{proof}

In fact, the key to the proof is that the maximal nilpotent orbits occuring in the Shalika germ expansion of $\mu_{G\cdot \Gamma}$ are the maximal orbits of $\nil(\Gamma)$.  

In \cite{CiubotaruOkada2023}, Ciubotaru and Okada obtain a similar result  directly, by analysing the asymptotic nilpotent cone of the characters of $G_{x,r}/G_{x,r+}$ appearing in $\pi^{G_{x,r+}}$.  

\begin{remark}
One might ask if Theorem~\ref{T:nil=wf} could be extended to show that $\WF(\pi)$ is the union of the nilpotent supports of the maximal orbits occurring in the $\Gamma$-asymptotic expansion \eqref{E:KM}.  The answer is expected to be negative.  In the supercuspidal case, the key result is \cite[Cor 10.2.3(1)]{Spice2021}, which implies that this latter set of orbits (in $\setorbitGamma$)  corresponds exactly to $\WF(\pi^0)$ (in $\setorbit$ for $G^0=\Cent_G(\Gamma)^\circ$), where $\pi^0$ is the associated depth-zero supercuspidal representation of $G^0$.  Cheng-Chiang Tsai\footnote{private correspondence and forthcoming work} has  constructed explicit examples of supercuspidal representations where the wave front set does not follow such a pleasant inductive structure. In  effect, one expects that when substituting Shalika germ expansions into the $\Gamma$-asymptotic expansion, cancellations among coefficients may occur.
\end{remark}

While the proof of Theorem~\ref{T:nil=wf} entails some additional hypotheses on $\ratk$, a consequence of the main theorem of Section~\ref{S:posdepth} is that, for $G=\SL(2,\ratk)$, the conclusion of the theorem holds whenever the characteristic and residual characteristic of $\ratk$ are not $2$.  

\section{Nilpotent orbits and nilpotent cones of \texorpdfstring{$G=\SL(2,F)$}{SL(2,F)}}\label{S:sl2nil}

For the rest of this paper we suppose that $\bG=\SL(2)$ and $\g = \mathfrak{sl}(2,F)$.  In this section, we derive some additional properties of the nilpotent support of an element $\Gamma \in \g^*$. 
We identify $\g$ and $\g^*$ with the trace form.

There are five nilpotent orbits: the zero orbit, and four two-dimensional principal (or regular) orbits that are in bijection with  the rational square classes $F^\times/(F^\times)^2$.  
Representatives of these five orbits in $\g$ are
\begin{equation}\label{E:nilpotentorbits}
\dot{X}_u=\mat{0&u\\0&0}
\end{equation}
where $u$ runs over the set $\{0,1,\ep,\varpi,\ep\varpi\}$ modulo $(F^\times)^2$ and $\ep\in \RR^\times$ is a fixed nonsquare.  For each $u$, write $\orb_u$ for the orbit in $\g^*$ corresponding to $\dot{X}_u$.
The following proposition relaxes the conditions for identifying the orbits in the nilpotent support of an element $\Gamma$.

\begin{proposition}\label{P:asymptotic_orbit_invariance}
Let $\g=\sltwo$ and $\Gamma \in \g^*\setminus\{0\}$.  Set $r=d(\Gamma)\in \mathbb{R}\cup \{\infty\}$.   Then 
\begin{enumerate}[(a)]
\item every $\orb\in\nil(\Gamma)$ meets $\Gamma+\g^*_{x,r}$ for some $x\in \buil(G)$ such that $d_x(\Gamma)<r$; 

\item  for each $x\in\buil(G)$ such that $d_x(\Gamma)<r$, if $\Gamma+\g^*_{x,r}$ meets a nilpotent orbit $\orb$, then $\orb\in\nil(\Gamma)$;

\item for each $x\in \buil(G)$, $\nil(\Gamma) = \{ \orb({}^g\Gamma, x) \mid g\in G\}=\nil_x(\Gamma)$, that is, every nonzero nilpotent orbit in $\cone(\Gamma)$ 
appears in the local nilpotent support at every $x$.
\end{enumerate}
\end{proposition}

\begin{proof}
The first two statements use that there are no closure relations between the principal orbits of $\sltwo$, and so the uniqueness of the minimal nilpotent orbit meeting any degenerate coset implies that any nontrivial degenerate coset meets only one nilpotent orbit.   

For (a), suppose $\orb \in \nil(\Gamma)$; then $\orb= \orb(\Gamma,x)$ for some $x\in \buil(G)$, implying $d_x(\Gamma)<r$.  Since $\Gamma\in \g^*_{r}\subset \g^*_{x,r}+\nilp^*$, the set $\Gamma+\g^*_{x,r}$ contains a (nonzero) nilpotent element $Y$.  Since $Y\in \Gamma+\g^*_{x,r}\subset \Gamma+\g^*_{x,d_x(\Gamma)}$, it lies in $\orb$, so $\orb$ meets the smaller coset, as required.

For (b), note that  if $d_x(\Gamma)<r$ then 
$0\notin \Gamma+\g^*_{x,r}\subset \Gamma+\g^*_{x,d(\Gamma)+}$;  any nilpotent orbit meeting the smaller set meets the larger one, and thus by uniqueness this orbit is $\orb(\Gamma,x)\in\nil(\Gamma)$.  

To prove (c), let $x\in \buil(G)$ and let $\nil_{x}(\Gamma)$ be the local nilpotent support of $\Gamma$ at $x$; we have already noted that $\nil_x(\Gamma)\subset \nil(\Gamma)$. 
The reverse inclusion follows from the one-dimensionality of $\buil(G)$. 
Let $\orb \in \nil(\Gamma)$; then $\orb = \orb(\Gamma,y)$  for some $y\in \buil(G)$.
Let $S$ be a split torus with associated root system $\Phi(G,S)=\{\pm \alpha\}$ such that $y\in \apart(G,S)$.    

Set $d=d_y(\Gamma)$ and let $\dot{\Gamma}\in \g$ correspond to $\Gamma$ via the trace form.  Choose $\dot{X}\in \orb$ such that $\dot{\Gamma}\in \dot{X}+\g_{y,d+}$. Conjugating both $\dot{\Gamma}$ and $\dot{X}$ by $G_y$ as necessary we may assume $\dot{X}\in \g_\alpha$. 
 Relative to the pinning of a fixed base point, we have the decomposition of $\RR$-modules
  $$
\g_{y,d} = \g_{-\alpha, d+\alpha(y)}\oplus \LieS_d \oplus \g_{\alpha,d-\alpha(y)}.
$$
Let $\alpha^\vee$ denote the positive coroot, and choose $g\in G$ so that 
$gx\in \apart(G,S)$ and $gx=y-\ell\alpha^\vee$ for some $\ell\geq 0$. Therefore if $d' = d-2\ell$ then $\g_{y,d}\subset \g_{gx,d'}$.
Since $\dot{X}\in \g_{\alpha,d-\alpha(y)}\setminus \g_{\alpha,(d-\alpha(y))+}$ and $d'-\alpha(gx)=d-\alpha(y)$, we conclude  $d_{gx}(\dot{\Gamma})=d_{gx}(\dot{X})=d'$ and $\dot{\Gamma}-\dot{X}\in \g_{gx,d'+}$.  By uniqueness, we infer that $\orb = \orb(\dot{\Gamma},gx)=\orb({}^{g^{-1}} \dot{\Gamma},x) \in \nil_x(\dot{\Gamma})$, yielding the result.
\end{proof}

We next determine $\nil(\Gamma)$ explicitly, for any $\Gamma \in \g=\sltwo$ (identified with its dual via the trace form).  There is nothing to do if $\Gamma$ is nilpotent.  If $\Gamma \neq 0$ is semisimple, then it is  $G$-conjugate to a matrix of the form
\begin{equation} \label{E:form}
\algy(u,v)=\mat{0&u\\v&0},
\end{equation}
for some $u,v\in\ratk^\times$.
Its centralizer is a maximal torus.   
 There is one $G$-conjugacy class of split torus, represented by any diagonal element, and two classes of unramified anisotropic tori, represented by $\algy(1,\ep)\in \g$ and $\algy(\varpi^{-1},\ep\varpi)\in\g$, respectively.  The classes of ramified tori are represented by $\algy(1,t)\in \g$ with $t\in \{\varpi, \ep\varpi, \ep^2\varpi, \ep^3\varpi\}$, noting that if $-\ep\in \ratk^2$ then there are only two classes. 

We can now describe the nilpotent support of each such element, using the parametrization given in \eqref{E:nilpotentorbits}.
 
\begin{lemma}\label{L:NilGamma}
Let $G=\SL(2,\ratk)$ and $\Gamma\in \g\setminus\{0\}$ semisimple.  
   If $\Gamma$ splits over $\ratk$, then $$\nil(\Gamma)=\{\orb_1, \orb_\ep, \orb_\varpi, \orb_{\ep\varpi}\}.$$ Otherwise,  $\Gamma$ is conjugate to $\algy(u,v)$ for some $u,v \in \ratk^\times$, and 
splits over $\extk = \ratk[\sqrt{uv}]$.  Let $\Norm_{\extk/\ratk}(\extk^\times)/(\ratk^\times)^2$ be represented by $\{1,\gamma\}$.  Then $u$ and $v$ are uniquely defined mod $\Norm_{\extk/\ratk}(\extk^\times)$ and
$$
\nil(\Gamma) = \{\orb_u, \orb_{u\gamma}\}.
$$
\end{lemma}

\begin{proof}
By Proposition~\ref{P:asymptotic_orbit_invariance}, we may fix the choice $x=x_0\in\buil(G)$ to be the vertex such that $\g_{x,r}$ is the set of traceless $2\times 2$ matrices with entries in $\PP^{\lceil r \rceil}$, and replace $\Gamma$ by any $G$-conjugate.

First suppose $\Gamma=\diag(a,-a)$ with $\val(a)=r$.  Let $u \in F^\times$ and note that if 
$
g_u=\smat{1&-\frac12a^{-1}u\\0&1}\in G
$ then ${}^{g_u}\Gamma = \smat{a&u\\0&-a}.$
 Therefore, for any $u$ such that $\val(u)=d<r$, we have ${}^{g_u} \Gamma \in \algy_u+\g_{x,d+}$.  Thus $\nil(\Gamma)$ contains every nonzero nilpotent orbit. 

Now suppose $\Gamma=\algy(u,v)$ for some $u,v\in F^\times$ such that $uv\notin (F^\times)^2$ and set  $E=F[\sqrt{uv}]$.
We calculate directly that the upper triangular entry of any $G$-conjugate of $\Gamma$ takes the form 
$$
u'=a^2 u - b^2 v = u(a^2-b^2vu^{-1}) \in u\Norm_{E/F}(E^\times) 
$$
for some $a,b\in F$, not both zero, from which it follows that $\nil(\Gamma)\subset \{\orb_u,\orb_{u\gamma}\}$.

For the reverse inclusion, first note that $\algy(u,v)$ is $G$-conjugate to $\algy(u\p^{-2n},v\p^{2n})$ for all $n\in \Z$ and for $n$ sufficiently large $\algy(u\p^{-2n},v\p^{2n})-\algy_{u\p^{-2n}}\in \g_{x,r}$.  Thus $\orb_u\in \nil(\Gamma)$.

Now note that when $E$ is ramified, we may take $\gamma = -uv$ so $\orb_{u\gamma}=\orb_{-v}$; since $\algy(u,v)$ is $G$-conjugate to $\algy(-v,-u)$ we are done by the preceding.  
If $E$ is unramified, we have instead $\gamma=uv$, whence $\orb_{u\gamma}=\orb_{v}$.   As $-1$ is a norm, we may choose $\alpha,\beta\in \ratk$ such that
$-1 = \beta^2-\alpha^2uv^{-1}$; then  $g  = \smat{\alpha & \beta \\ \beta & \alpha uv^{-1}} \in G$ satisfies ${}^g\algy(u,v) = \algy(v,u)$, and again by the preceding we may conclude $\orb_v \in \nil(\Gamma)$.
\end{proof}

\section{Representations of \texorpdfstring{$G_x$}{Gx} associated to nilpotent orbits} \label{S:Representations} 

\subsection{Shalika's representations of \texorpdfstring{$\SL(2,\RR)$}{SL(2,R)}}
\label{SS:Shalika}

In his thesis, Shalika constructed all irreducible representations of $K=\SL(2,\RR)$.  In this section we recap his explicit construction for the so-called ramified case, which attaches an irreducible representation of $K$ to certain $K$-orbits in $\g^*$; we'll then provide a coordinate-free generalization more suited to our needs in the next section.

Let $S$ be the diagonal split torus, $B$ the upper triangular Borel subgroup and $U$ its unipotent radical.  We use a subscript $0$ to indicate their intersections with $K$: $S_0=S\cap K$, $B_0=B\cap K$ and $U_0=U\cap K$. Let $x_0\in \apart(G,S)$ be such that $K=G_{x_0}$ and $z_0$ the barycentre of the positive alcove adjacent to $x_0$ (relative to $B$).

Let $d$ be a positive integer.  Choose $u \in \PP^{-d}\setminus \PP^{-d+1}$ and nonzero $v\in \PP^{-d+1}$ and consider the anti-diagonal matrix $\algy := \algy(u,v)\in \g_{x_0,-d}$ of \eqref{E:form}.
Identify this with the element $\dualy\in \g^*_{x_0,-d}$ by the rule $\dualy(Z)=\tr(\algy Z)$ for all $Z\in\g$.  If $v=0$ then $\dualy$ is nilpotent and its centralizer $C_K(\dualy)$ in $K$ coincides with $ZU_0$, where $Z=\{\pm I\}$.  Otherwise, $\dualy$ is semisimple and $C_K(\dualy)$ is a torus.  Note that every $\dualy\in \g^*_{x_0,-d}$ that represents a degenerate coset is $K$-conjugate to one of this form.

Define an open subgroup of $K$ by 
$$
J_{d} = \mat{1+\PP^{\lceil d/2 \rceil}&\PP^{\lceil d/2 \rceil} \\ \PP^{\lceil (d+1)/2\rceil} & 1+\PP^{\lceil d/2 \rceil}} \cap K.
$$
It is straightforward to verify that $\dualy$ gives a well-defined character $\eta_\dualy$ of $J_d$, trivial on $G_{x_0,d+}$, by the rule 
\begin{equation}\label{E:formula}
\eta_\dualy(g)=\psi(\tr(\algy(g-I))).
\end{equation}
This character is trivial on $K_{d+}$ and depends only on the classes $u+\PP^{\lceil (-d+1)/2\rceil}$ and $v+\PP^{\lceil -d/2 \rceil}$.
For any choice of character $\theta$ of $C_K(\dualy)$ agreeing with $\eta_\dualy(g)$ on $C_K(\dualy)\cap J_d$, 
write $\eta(\dualy,\theta)$ for the resulting extension to a character of $C_K(\dualy)J_d$.  

Then Shalika proves the following result with an intricate elementary argument \cite[Thm 4.2.1, Thm 4.2.5, \S4.3]{Shalika1966}.  

\begin{proposition} \label{P:Shalika}
Let $\dualy\in \g^*_{x_0,-d}$ be as above, and $\theta$ any character of its centralizer $C_K(\dualy)$ agreeing with $\eta_\dualy$ on $C_K(\dualy)\cap J_d$. Then the representation
$$
\Sh_{x_0}(\dualy,\theta) = \Ind_{C_K(\dualy)J_d}^K \eta(\dualy,\theta)
$$
is irreducible and independent (up to equivalence) of the choice of representative in the $K$-orbit of $\algy(u+\PP^{\lfloor(d+1)/2\rfloor},v+\PP^{\lceil (d+1)/2\rceil})$.  It is of degree $\frac12 q^{d-1}(q^2-1)$ and of depth $d$, meaning it is nontrivial on $K_d$ but trivial on $K_{d+}$.
\end{proposition}

\subsection{Irreducible representations of \texorpdfstring{$G_x$}{Gx} parametrized by degenerate cosets at \texorpdfstring{$x$}{x}} \label{S:gx}

Our goal in this section is to 
give a coordinate-free interpretation of Shalika's construction that allows us to unambiguously attach representations of $G_x$ to any degenerate coset of negative depth. 

Note that $\GL(2,\ratk)$ acts on $\buil(G)$, and all vertices are conjugate under this action.  This conjugacy does not in general preserve the $\SL(2,\ratk)$-orbit of $\Gamma$ or $X$.  

\begin{example}\label{E:omegaeta}
Let $x_0, z_0$ be as in Section~\ref{SS:Shalika} and $x_1$ the other vertex of the chamber containing $x_0$ in its closure.
The element $\omega = \smat{0 & 1 \\ \varpi & 0}$ used in \cite{Nevins2005} is an affine reflection such that $\omega\cdot x_0 = x_1$, and ${}^\omega \algy(u,v) = \algy(\varpi^{-1}v, \varpi u)$.  Thus in particular in the case of nilpotent orbits, where $\algy(0,1)\sim \algy(-1,0)$, we have ${}^\omega \cO_1 = \cO_{-\varpi}$.
On the other hand, the element $\eta = \smat{1 & 0\\0&\varpi}$ used in \cite{Nevins2013} is a translation such that $\eta\cdot x_0 = x_1$, but now ${}^\eta \algy(u,v) = \algy(\varpi^{-1}u,\varpi v)$ so  ${}^\eta \cO_1 = \cO_{\varpi}$
instead.
\end{example}

We begin by showing that any degenerate coset determines a chamber of $\buil(G)$ adjacent to $x$.  

\begin{lemma}\label{L:CGamma}
Let $G=\SL(2,\ratk)$.  Let $x\in \buil(G)$ be any vertex and let $\Gamma \in \g^*_{x,-d}\setminus \g^*_{x,-d+}$ represent a degenerate coset for some  $d>0$.  
Then there exists a unique chamber $\mathcal{C}=\mathcal{C}_\Gamma$ of $\buil(G)$ adjacent to $x$, independent of the choice of representative of $\Gamma+\g^*_{x,-d+}$,  such that for any $z\in \mathcal{C}$ we have $\Gamma \in \g^*_{x,-d}\cap\g^*_{z,-d+}$.  Moreover, we have $\Cent_{G_x}(\Gamma)=\Cent_{G_z}(\Gamma)$.
\end{lemma}

\begin{proof}
Uniqueness is immediate:  given $z'$ in any other chamber adjacent to $x$, the geodesic from $z$ to $z'$ contains $x$; hence $\g^*_{z,-d+}\cap \g^*_{z',-d+}\subset \g^*_{x,-d+}$, so does not contain $\Gamma$. 
Identify $\Gamma$ with an element $\dot{\Gamma}\in \g_{x,-d}$ via the trace form.  Choose a nilpotent element $\dot{X}\in \dot{\Gamma}+\g_{x,-d+}$.  By \cite[\S 5]{DeBackerNilpotent2002}, we may complete $\dot{X}$ to an $\sltwo$-triple $\{\dot{X}, \dot{H}\in \g_{x,0}, \dot{Y}\in \g_{x,d}\}$ and find a split torus $S$ and corresponding apartment $\apart(G,S)$ containing $x$, such that if $\Phi(G,S)=\{\pm \alpha\}$, then $\dot{X}\in \g_{\alpha}$ and $\dot{Y}\in \g_{-\alpha}$.  
Let $\mathcal{C}$ be the positive alcove adjacent to $x$ in this apartment.
 
Note that we have $\Cent_{\g}(\dot{Y}) = \g_{-\alpha}$. From \cite[Lemma 5.2.1]{DeBackerNilpotent2002} we know that 
$$
\dot{X}+\g_{x,-d+} = {}^{G_{x,0+}}\left(\dot{X} + \Cent_{\g_{x,-d+}}(\dot{Y}) \right);
$$
thus there exists $g\in G_{x,0+}$ such that $\dot{\Gamma}\in {}^g(\dot{X}+ \g_{-\alpha}\cap \g_{x,-d+})$.  Since $G_{x,0+}$ fixes $C$ and the coset $\dot{\Gamma}+\g_{x,-d+}$, we may without loss of generality replace the Lie triple and torus of the preceding paragraph with their $g$-conjugate, so that we have    $\dot{\Gamma} \in \dot{X}+\g_{-\alpha}\cap \g_{x,-d+}$.  
For any $z\in C$ we have $0 < \alpha(z-x) < 1$; thus since $\alpha(x), d\in \Z$ we may conclude
$$
\g_\alpha \cap \g_{x,-d} = \g_\alpha \cap \g_{z,-d+}.
 \quad \text{and}
\quad \g_{-\alpha}\cap \g_{x,-d+} = \g_{-\alpha}\cap \g_{z,-d+}.
$$
Since $\dot{\Gamma}$ lies in the sum of these two spaces we have
 $\dot{\Gamma}\in \g_{z,-d+}$, whence
$\Gamma \in \g^*_{x,-d}\cap \g^*_{z,-d+}$.  

Finally, note that $\Cent_{G}(\dot{X})=U_\alpha$ and $U_\alpha\cap G_x=U_\alpha\cap G_z$.  Since $\dot{\Gamma}\in\dot{X}+\g_{x,-d+}$, we have $\Cent_{G_x}(\dot{\Gamma}) \subset \Cent_{G_x}(\dot{X})G_{x,0+}=\Cent_{G_z}(\dot{X})G_{x,0+}\subset G_z$. 
\end{proof}

\begin{definition}\label{D:defJ}
Let $d=-d_x(\Gamma)$ be such that $\Gamma+\g^*_{x,-d+}$ is a degenerate coset.  Let $z$ be the barycentre of the associated alcove $\mathcal{C}_\Gamma$.  Define the subgroup
\begin{equation}\label{E:defJ}
J_{x,\Gamma}= \begin{cases}
G_{x,d/2} & \text{if $d=d_x(\Gamma)$ is odd};\\
G_{z,d/2} & \text{if $d$ is even}.
\end{cases}
\end{equation}
\end{definition}

Note that when $x=x_0$ and $z=z_0$ we have $J_{x,\Gamma}=J_d$.

Since $G_{x,n+}\subseteq G_{z,n} \subseteq G_{x,n}$ for any integer $n$, it follows directly that for all $d$, we have
$$
G_{x,d/2+}\subseteq J_{x,\Gamma} \subseteq G_{x,d/2}.
$$
Since $\Gamma \in \g^*_{x,-d}\cap \g^*_{z,-d+}$, it defines a character $\eta_{\Gamma}$ of $J_{x,\Gamma}$ that is trivial on $G_{x,d+}$ via the corresponding Moy--Prasad isomorphism.  
 The character depends only on the coset $\Gamma + \g_{x,-d/2}^*$ if $d$ is odd and on $\Gamma+\g^*_{z,-d/2+}$ otherwise.  Moreover, since $\Cent_{G_x}(\Gamma) = \Cent_{G_z}(\Gamma)$ we deduce directly that
$J_{x,\Gamma}$ is normalized by $C_x(\Gamma):=\Cent_{G_x}(\Gamma)$.

Thus, for any character $\theta$ of $C_x(\Gamma)$ coinciding with $\eta_{\Gamma}$ on the intersection of their domains there is a unique extension  $\eta(\Gamma,\theta)$ of $\eta_\Gamma$ to  $C_x(\Gamma)J_{x,\Gamma}$. Define
$$
\Sh_x(\Gamma,\theta) = \Ind_{C_x(\Gamma)J_{x,\Gamma}}^{G_x} \eta(\Gamma,\theta).
$$

\begin{proposition}\label{P:Gxreps}
Suppose $\Gamma$ represents a degenerate coset at a vertex $x\in \buil(G)$ and $-d=d_x(\Gamma)<0$.  Suppose $\theta$ is a character of the centralizer $C_x(\Gamma)$ of $\Gamma$ in $G_x$ defining a character $\eta(\Gamma,\theta)$ of  $C_x(\Gamma)J_{x,\Gamma}$. Then
\begin{enumerate}[(a)]
\item   $\Sh_x(\Gamma,\theta)$ is  an irreducible representation of $G_x$ of depth $d$ and degree $\frac12 q^{d-1}(q^2-1)$;  
\item  $\Sh_x(\Gamma,\theta) \cong \Sh_x(\Gamma',\theta')$ if and only if there exists $g\in G_x$ such that $\eta(\Gamma,\theta)={}^g\eta(\Gamma',\theta')$; and 
\item for any $\nu\in \GL(2,\ratk)$ we have 
\begin{equation}\label{Eq:GL2conj}
{}^\nu \Sh_{x}(\Gamma,\theta) \cong \Sh_{\nu\cdot x}({}^\nu\Gamma, {}^\nu \theta).
\end{equation}
\end{enumerate}
\end{proposition}

\begin{proof}
When $x=x_0$ and $\Gamma\in \g^*$ corresponds to some $\algy(u,v)\in\g_{x_0,-d}\setminus \g_{x_0,-d+}$, then this construction coincides with Shalika's.    If $g\in G_{x}$, then ${}^gC_{x}(\Gamma) = C_{x}({}^g\Gamma)$ and ${}^gJ_{x,\Gamma} = J_{x,{}^g\Gamma}$, so we obtain the invariance of $\Sh_x(\Gamma,\theta)$ under $G_x$-conjugacy and the choice of representative of the appropriate coset of $\Gamma$.
More generally, for any $\nu\in \GL(2,\ratk)$ such that $\nu\cdot x_0 = x$, we have ${}^\nu(\g^*_{x_0,d})=\g^*_{x,d}$, ${}^\nu C_{{x_0}}(\Gamma) = C_{x}({}^\nu\Gamma)$ and ${}^\nu J_{x_0,\Gamma} = J_{x,{}^\nu\Gamma}$. Thus
$$
{}^\nu \Sh_{x_0}(\Gamma,\theta) \cong \Sh_{x}({}^\nu\Gamma, {}^\nu \theta),
$$
where we have identified a $\nu$-conjugate of a representation of $G_{x_0}$ with a representation of $G_x$ under the group isomorphism ${}^\nu G_{x_0}\cong G_x$.  Since $\GL(2,\ratk)$ acts transitively on the set of vertices of $\buil(\SL(2,\ratk))$, the rest of the statements follow from  Proposition~\ref{P:Shalika}. 
\end{proof}

The simple nature of the representations $\Sh_x(\Gamma,\theta)$  is revealed as follows.

\begin{lemma}\label{L:Shalikaequiv}
Suppose $x$ is a vertex of $\buil(G)$ and $\Gamma_1, \Gamma_2 \in \g^*_{x,-d}$ represent nonzero but degenerate cosets of $\g^*_{x,-d}/\g^*_{x,-d+}$ for some $d>0$. Suppose $s \in \real$ satisfies $\Gamma_1\in \Gamma_2 + \g^*_{x,-s}$.  Then
for any choice of characters $\theta_i$ of $C_x(\Gamma_i)$ such that the characters $\eta(\Gamma_i,\theta_i)$ agree upon restriction to $C_x(\Gamma_i)J\cap G_{x,s+}$ for $i\in \{1,2\}$, we have
\begin{equation}\label{E:restrictionsame}
\Res_{G_{x,s+}}\Sh_x(\Gamma_1,\theta_1)\cong \Res_{G_{x,s+}}\Sh_x(\Gamma_2,\theta_2).
\end{equation}
In particular, if $s\geq d/2$ 
then \eqref{E:restrictionsame} holds independent of $\theta_i$. 
\end{lemma}

\begin{proof}
For any $\Gamma_i$, the two representations have the same degree $\frac12q^{d-1}(q^2-1)$ and the same depth $d$.  If $s \geq d$ then both sides are $1$-isotypic of the same degree hence equivalent.   

Suppose $s<d$. Since $\Gamma_1\in \Gamma_2 + \g^*_{x,-s}$,  we have $C_x(\Gamma_1) \subset C_x(\Gamma_2)G_{x,d-s}$. Since $\Gamma_1\in \Gamma_2+\g^*_{x,-d+}$,  Lemma~\ref{L:CGamma} yields $J_{x,\Gamma_1}=J_{x,\Gamma_2}$; let us denote this group $J$.  Thus $\eta_{\Gamma_i}$ for $i\in \{1,2\}$ are characters of $J$  that agree on $J\cap G_{x,s+}$.

If $s\geq d/2$ then 
$G_{x,s+}\subset J$, and so $\Res_{G_{x,s+}\cap C_x(\Gamma_i)J}\; \eta(\Gamma_i,\theta_i)=\eta_{\Gamma_i}$ is independent of $\theta_i$.  
Mackey theory thus yields the decomposition
\begin{equation}\label{E:etai}
\Res_{G_{x,s+}}\Sh_x(\Gamma_i,\theta_i) \cong \bigoplus_{\gamma \in G_{x}/C_x(\Gamma_i)J} {}^\gamma\eta_{\Gamma_i}|_{G_{x,s+}}.
\end{equation}
Each $\gamma \in C_x(\Gamma_i)G_{x,d-s}/C_x(\Gamma_i)J$ fixes the character $\eta_{\Gamma_i}|_{G_{x,s+}}$.  The elements $\gamma' \in G_x/C_x(\Gamma_1)G_{x,d-s}=G_x/C_x(\Gamma_2)G_{x,d-s}$ parametrize the 
orbit of the coset $\Gamma_1+\g^*_{x,-s} = \Gamma_2+\g^*_{x,-s}$.  Thus \eqref{E:etai} gives the same sum of characters for $i\in \{1,2\}$.

If instead $s<d/2$, then $G_{x,d-s}\subseteq J$ so $C_x(\Gamma_1)J=C_x(\Gamma_2)J$.  Since $J\subseteq G_{x,s+}$, the double coset space $G_{x,s+}\backslash G_{x} /C_x(\Gamma_i) J$ is now equal to $G_x/C_x(\Gamma_i)G_{x,s+}$, and is independent of $i$. So again by Mackey theory we have
\begin{align*}
\Res_{G_{x,s+}}\Sh_x(\Gamma_i,\theta_i) 
&= \bigoplus_{\gamma \in G_x/C_x(\Gamma_i)G_{x,s+}} 
\Ind_{G_{x,s+}\cap {}^\gamma(C_x(\Gamma_i) J)}^{G_{x,s+}} {}^\gamma(\eta(\Gamma_i,\theta_i))\\
&=\bigoplus_{\gamma \in G_x/C_x(\Gamma_i)G_{x,s+}}{}^\gamma\left(\Ind_{G_{x,s+}\cap C_x(\Gamma_i) J}^{G_{x,s+}} \eta(\Gamma_i,\theta_i))\right).
\end{align*}
When the restriction of $\eta(\Gamma_i,\theta_i)$ to $G_{x,s+}\cap C_x(\Gamma_1) J=G_{x,s+}\cap C_x(\Gamma_2) J$ is independent of $i$, we infer \eqref{E:restrictionsame}.
\end{proof}

\subsection{Representations attached to nilpotent orbits}\label{attach}

Let $X\in \nilp^*\setminus\{0\}$ and let $\lambda$ be a corresponding adapted one-parameter subgroup, whose centralizer in $G$ is a maximal split torus $S$.  In fact, $S$ is generated by $S_0$ and $\lambda(\varpi)$, and $\Cent_G(X)=ZU$ where $Z$ is the center of $G$ and $B=SU$ is a Borel subgroup.  
For any vertex $x\in\buil(G)$,  applying the Cartan decomposition yields 
\begin{equation}\label{E:GGxorbit}
\cO = G\cdot X = \bigsqcup_{n\in \Z}G_x \cdot (\lambda(\varpi)^n\cdot X) = \bigsqcup_{n\in \Z}G_x \cdot (\varpi^{2n} X)
\end{equation}
as the decomposition of the $G$-orbit of $X$ into disjoint $G_x$-orbits.  

\begin{proposition}
Let $x$ be a vertex in $\buil(\bG,F)$, $\cO$  a nonzero nilpotent $G$-orbit in $\g^*$ and $\zeta$ a character of $Z$.  Let $\{X_{-d}\mid d_x(X)=-d<0\}$ be any set of representatives of the $G_x$-orbits in $\cO \setminus \g^*_{x,0}$.  Then the \emph{representation of $G_x$ attached to $\cO$ with central character $\zeta$}, given by
\begin{equation}\label{E:deftau}
\tau_x(\cO,\zeta) = \bigoplus_{d>0} \Sh_x(X_{-d},\zeta),
\end{equation}
is independent of choices up to $G_x$-equivalence.
\end{proposition}

\begin{proof}
The chamber $\mathcal{C}_X$ associated to $(X_{-d},x)$ by Lemma~\ref{L:CGamma}    defines an Iwahori subgroup of $G_x$ with pro-p unipotent radical $U_x$; by construction, we have $C_{x}(X) = ZU_x$.  Since $\eta_X$ is trivial on $ZU_x\cap J_{x,X}$, the character $\zeta$ of $ZU_x$ defined by $\zeta(zu)=\zeta(z)$ for all $z\in Z$ and $u\in U_x$ extends
$\eta_X$.  Thus Proposition~\ref{P:Gxreps} applies.
\end{proof}

It follows from \eqref{E:GGxorbit} that the parity of $d_x(Y)$, for any $Y\in \cO$, is an invariant of the $G$-orbit, and we call this the \emph{parity depth of $\cO$ at $x$}.
  Therefore the depths $d$ of the components of $\tau_x(\cO,\zeta)$ all have parity equal to the parity depth of $\cO$ at $x$.  
  
  Note that the restriction of $\tau_x(\cO,\zeta)$ to any subgroup of $G_{x,0+}$ is independent of the choice of $\zeta$, so we may drop $\zeta$ from the notation in such cases.  
As needed, we associate to the zero nilpotent orbit the trivial representation of $G_x$, and denote it $\tau_x(\{0\})$.

\section{The case of positive-depth representations of \texorpdfstring{$\SL(2,\ratk)$}{SL2F}}\label{S:posdepth}

We briefly recap the classification of irreducible representations $\pi = \pi(\chi,\Gamma)$ of $\SL(2,\ratk)$ of positive depth, then establish that their  explicit branching to a maximal compact open subgroup $G_x$ can be described as twists of the datum $(\chi,\Gamma)$ defining $\pi$.  This allows us to state and prove our main theorem in this case, and to explicitly compute the constant terms that arise.

\subsection{Representation of \texorpdfstring{$\SL(2,\ratk)$}{SL(2,F)} of positive depth}

The classification of the irreducible admissible representations of $G=\SL(2,\ratk)$ was given in Shalika's 1966 thesis \cite{Shalika1966}, and Sally and Shalika obtained many explicit results on their characters in several papers, including \cite{SallyShalika1984}.  An excellent overview is given \cite{AdlerDeBackerSallySpice2011}.  In this section we address the positive-depth supercuspidal representations, using the parametrization of Adler and Yu \cite{Adler1998, Yu2001,FintzenFixYu2021}.  Because the tori in $\SL(2,\ratk)$ are one-dimensional, the correcting twist to this construction given by Fintzen, Kaletha and Spice in  \cite[Definition 3.1]{FintzenKalethaSpice2021} is trivial in this case.
%

\begin{proposition}\label{P:posdepthpara}
The isomorphism classes of irreducible representations of $\SL(2,\ratk)$ of positive depth $r$ are parametrized by the $G$-conjugacy classes of pairs $(T,\chi)$, where $T$ is a maximal torus of $G$ and $\chi$ is a character of $T$ of depth $r$.  
\end{proposition}

This is equivalent to the well-known classification of the irreducible positive depth representations of $\SL(2,\ratk)$ in terms of unrefined minimal $K$-types.  To construct the representations explicitly  we first recall some facts about the maximal tori and their characters.

Let  $T$ be a maximal torus of $G$ 
and let $\chi$ be a character of $T$ of depth $r>0$.  The building $\buil(T)$ of $T$ embeds into $\buil(G)$ as the apartment $\apart(G,T)$ if $T$ is split and as a single point $\{x_T\}$ otherwise, which is a vertex if $T$ is unramified and the midpoint of a chamber if $T$ is ramified.  It follows that the depth $r$ is an integer if $T$ splits over an unramified extension and an element of $\frac12 + \Z$ otherwise.

To each pair $(T,\chi)$ we associate an element $\Gamma$ as follows.  If $\LieT$ denotes the Lie algebra of $T$, then via the Moy-Prasad isomorphism $e\colon \LieT_{r/2+}/\LieT_{r+} \to T_{r/2+}/T_{r+}$ there exists a nonzero element $\Gamma =\Gamma_\pi \in \LieT^*_{-r}$, uniquely defined modulo $\LieT^*_{-r/2}$, such that
$$
\chi(t)  = \psi(\Gamma(e^{-1}(t))).
$$
We identify $\Gamma$ with an element of $\g^*$ that is zero on the $T$-invariant complement of $\LieT$ in $\g$.  Then $\Gamma\in \g^*_{x,-r}$ for any $x\in \buil(T)$ and we recover $T$ as $\Cent_G(\Gamma)$.  Moreover, $\Gamma$ thus defines a character of $G_{x,r}/G_{x,r+}\cong \g_{x,r}/\g_{x,r+}$, and following the work of Moy and Prasad, the pair $(G_{x,r},\Gamma)$ is called an \emph{unrefined minimal $K$-type}.

\begin{proof}[Proof of Proposition~\ref{P:posdepthpara}]
The classification is known; we only wish to construct the representations $\pi = \pi(T,\chi)$ variously as follows.

 If $T$ is a split torus, then choose a Borel subgroup $B=TN$ of $G$ containing $T$ and extend $\chi$ trivially across $N$ to a character of $B$.  Set
$$
\Ind_{TN}^G(\chi) = \{ f\colon G\to \C \mid f(tng)=\chi(t)\nu(t)f(g) \forall t\in T, n\in N, g\in G\}
$$
where  $\nu$ is the square root of the modular character and is given on  $T\cong F^\times$ by  the $p$-adic norm.
Then $\pi(T,\chi)=\Ind_B^G(\chi)$ is an irreducible principal series representation.

 If $T$ is anisotropic, with associated point $x_T\in\buil(G)$, then we first extend $\chi$ to a character of $TG_{x_T,r/2+}$, by setting
 $$
 \chi(tg) = \chi(t) \psi(\Gamma(e^{-1}(g)))
 $$
 where $e \colon \g_{x_T,r/2+}/\g_{x_T,r+} \to G_{x,r/2+}/G_{x,r+}$ is the Moy-Prasad isomorphism.  When $G_{x_T,r/2}\neq G_{x_T,r/2+}$ (which will happen only if $T$ is unramified and $r\in 2\Z$), we take a certain Weil-Heisenberg lift of $\chi|_{G_{x_T,r}}$ to form a  $q$-dimensional representation $\omega$ of $T\ltimes G_{x_T,r/2}$, and set $\kappa(tg) = \chi(t)\omega(t,g)$.  Then $\pi(T,\chi) = \cind_{TG_{x_T,r/2}}^G \kappa$ is an irreducible supercuspidal representation.
\end{proof}

Given $\pi=\pi(T,\chi)$, we let $\Gamma = \Gamma_\pi$ denote a choice of element in $\g^*$ realizing the character $\chi$, as preceded the proof.  Then since $T=\Cent_G(\Gamma)$ we may also say that $(\chi,\Gamma)$ is the data defining $\pi$.  

\subsection{Branching rules obtained as twists of the inducing datum}
We begin by proving that the branching rules obtained in \cite[Theorem 7.4]{Nevins2005} and \cite[Theorem 6.2]{Nevins2013} are in fact constructable from twists of the data $(\chi,\Gamma)$.

\begin{theorem}\label{T:posdepth}
Let $\pi=\pi(T,\chi)$ be an irreducible representation of $G$ of depth $r>0$.  Let $\Gamma=\Gamma_\pi\in \g^*$ realize $\chi$ as above. 
Then for any vertex $x\in\buil(G)$ we have
\begin{equation}\label{E:basicdecomp}
\Res_{G_{x}}\pi = \pi^{G_{x,r+}} \oplus \bigoplus_{g\in \allbutone} \Sh_x({}^g\Gamma, {}^g\chi)
\end{equation}
where $\allbutone$ denotes a parameter set for the $G_x$-orbits in $G\cdot \Gamma$  that do not meet $\g^*_{x,-r}$, that is, such that the  coset ${}^g\Gamma + \g^*_{x,d_x({}^g\Gamma)+}$ is degenerate.
\end{theorem}

\begin{proof}
We begin with the case that $T=\Cent(\Gamma)$ is anisotropic. Set $y=x_T$ and let  $\pi(T,\chi) =\cind_{TG_{y,r/2}}^G \kappa$ be the corresponding  supercuspidal representation.  

First suppose that we are in the special case that $x=x_0$ and $y \in \overline{\mathcal{C}}$, the  closure of the fundamental alcove, which was the case considered in \cite{Nevins2013}.  By \cite[Prop 4.4]{Nevins2013},
the double coset space $G_x\backslash G/TG_{y,r/2}$ that arises in the Mackey decomposition
$$
\Res_{G_x}\pi = \bigoplus_{g \in G_x\backslash G/TG_{y,r/2}}\Ind_{G_x\cap {}^g(TG_{y,r/2})}^{G_x} {}^g\kappa
$$
is independent of $r$ and is given by $G_x\backslash G/T$.  Since $T=\Cent_G(\Gamma)$, this latter space parametrizes the $G_x$-orbits in the $G$-orbit of $\Gamma$ in $\g^*$.  By \cite[Thm 6.1]{Nevins2013}, each of these Mackey components is irreducible. 

Now since $\Gamma$ has depth $-r$ and depth is $G$-invariant, ${}^g\Gamma$ meets $\g^*_{x,-r}$ if and only if $d_x({}^g\Gamma)=-r$.  Since $\Cent_G(\Gamma)=T$, this happens if and only if $gx = x_T= y$ in which case the corresponding Mackey component has depth $r$ and so lies in $\pi^{G_{x,r+}}$.  Note that $\pi^{G_{x,r+}} \neq \{0\}$ thus arises only when $T$ is an unramified torus attached a vertex $y$ in  the same $G$-conjugacy class as $x$.

When $gx\neq y$, then by in \cite[Thm 6.2]{Nevins2013},\footnote{Errata to \cite[Thm 6.2]{Nevins2013}: the decomposition in case $y=1$ is missing the term corresponding to the double coset representative $\mathsf{e}^{\eta}$.}  the corresponding Mackey component satisfies 
$$
\Ind_{G_{x}\cap {}^g(TG_{y,r/2})}^{G_x} {}^g\kappa \cong \Sh_{x}({}^g\Gamma, {}^g\chi),
$$
as required, yielding \eqref{E:basicdecomp} for the fundamental case.

Now suppose that $x\in \buil(G)$ is an arbitrary vertex.  Then there exists $k\in \GL(2,\ratk)$ such that $k x = x_0$.  Choose $h\in \SL(2,\ratk)$ such that  $hy\in k^{-1} \overline{\mathcal{C}}$.  Then via the identification ${}^kG_{x}= G_{x_0}$, and then the isomorphism ${}^h\pi\cong \pi$, we may write
$$
\Res_{G_{x_0}}{}^{kh}\pi = \Res_{G_x}{}^h\pi \cong \Res_{G_x}\pi.
$$
Even when $kh\notin \SL(2,\ratk)$, the data defining the representation ${}^{kh}\pi$ is simply $({}^{kh} T,{}^{kh}\chi,{}^{kh}\Gamma, khy)$. We may therefore apply the decomposition \eqref{E:basicdecomp} to $\Res_{G_{x_0}}{}^{kh}\pi$.  
Since ${}^kG_{x,r+}=G_{x_0,r+}$, we have $({}^{kh}\pi)^{G_{x_0,r+}}=({}^h\pi)^{G_{x,r+}}$.
Moreover, for any $g$ defining a $G_{x_0}$-orbit of ${}^{kh}\Gamma$ that does not meet $\g_{x_0,-r}^*$, we have
$$
S_{x_0}({}^{gkh}\Gamma,{}^{gkh}\chi)
={}^k(S_{x_0}({}^{(k^{-1}gk) h}\Gamma,{}^{(k^{-1}gk)h}\chi))=
S_x({}^{g'h}\Gamma,{}^{g'h}\chi),
$$
where $g' = k^{-1}gk$; then $g'h$ defines a $G_x$-orbit of $\Gamma$ that does not meet $\g_{x,-r}^*$, showing the index sets correspond.  

We now consider the case that $T$ is a split torus, so that $\pi=\pi(T,\chi)=\Ind_B^G\chi$ for some Borel subgroup $B=TU$ containing $T$ having $U$ as its unipotent radical. Since $G=G_xB$, there is a unique (highly reducible) Mackey component in this case.  Instead, in the special case that $x\in \{x_0,x_1\}$ and $T=S$, the decomposition of $\Res_{G_x}\pi$ into irreducibles is found in \cite{Nevins2005} by explicitly decomposing the $G_x$-subrepresentations $\pi^{G_{x,n}}$ as $n\to \infty$.  We need to show that this decomposition is in fact of the form \eqref{E:basicdecomp}.  We begin with $x=x_0$ and $T=S$.

First note that as $\pi$ has depth $r$ at $x$, the subrepresentation $\pi^{G_{x,r+}}$ is nonzero, and in fact is irreducible by \cite[Prop 4.4]{Nevins2005} (where this space is denoted $V^{K_m}$, for $m=r+1$ the conductor of $\chi$).  By depth we know that ${}^g\Gamma \in \g_{x,-r}^*$ if and only if ${}^gT \subset G_x$, so the $G_x$-orbits in $G\cdot \Gamma$ meeting $\g_{x,-r}^*$ correspond to the trivial double coset of $G_{x}\backslash G/T$. 

The irreducible representations of $G_{x}$ of depth greater than $r$  appearing in $\Res_{G_x}\pi$ are classified in \cite[Thm 7.4]{Nevins2005}. The notation in that paper relates to ours as follows.  Identify $\g$ and $\g^*$ via the trace form.  The conductor $n$ is $d+1$, and for any $u\in \RR^\times, v\in \PP$ we have
\begin{equation}\label{NotationNevins2005}
\mathcal{D}_n(\rho, \dot{X}(u,v)) := \Sh_{x}(\varpi^{-d}\dot{X}(u,v), \rho).
\end{equation}
If $\dot{\Gamma}$ is  the diagonal matrix $\diag(a,-a)\in \g_{x,-r}$,  set $\gamma_0 = a\varpi^{d}$, $\gamma_1 = a\ep^{-1}\varpi^{d}$ and write
$$
g_i= \mat{1&  -\frac12 \gamma_i^{-1}\\ \gamma_i & \frac12}.
$$
Following the notation in \emph{loc. cit.}, define $Y_0 = \dot{X}(1,\gamma_0^2)$ and $Y_1 = \dot{X}(\ep,\ep\gamma_1^2)$;  then
\begin{equation}\label{E:Y0Y1}
 \varpi^{-d}Y_0
 ={}^{g_0}\Gamma \quad \text{and}\quad 
 \varpi^{-d}Y_1 
 ={}^{g_1}\Gamma.
\end{equation}
 It follows that 
 $\rho_i:={}^{g_i}\chi$ is a character of $C_{{x}}(Y_i)=C_{{x}}({}^{g_i}\Gamma)$ extending $\eta_{{}^{g_i}\Gamma}$.
 
Then, in our notation, \cite[Thm 7.4]{Nevins2005} asserts that for each integer $d>r$, $\Res_{G_x}\pi$ has two irreducible components of depth $d$, denoted $W_{d-1}^\pm$, and these are explicitly given by 
\begin{equation}\label{E:psdecomp}
W_{d-1}^+\oplus W_{d-1}^-\cong \Sh_{x}({}^{g_0}\Gamma, {}^{g_0}\chi) \oplus \Sh_{x}({}^{g_1}\Gamma, {}^{g_1}\chi).
\end{equation}
Noting the factorization 
$$
\mat{1 & -\frac12\gamma^{-1}\\ \gamma & \frac12} = \mat{1&0\\\gamma&1}\mat{1 & -\frac12 \gamma^{-1}\\0&1},
$$
we see that as $\gamma$ runs over the distinct square classes in $\gamma\in\PP^{d}\setminus \PP^{d+1}$, for all $d>0$, we   obtain representatives of the distinct nontrivial cosets of $G_{x}\backslash G/T$, which is the index set $\allbutone$.  


The same result is obtained for $x=x_1$ in \cite[Cor 4.6, Thm 7.4]{Nevins2005} via conjugation by $\omega = \smat{0 & 1 \\ \varpi & 0} \in \GL(2,\ratk)$, as $\omega x_0=x_1$.   For arbitrary $x$, we proceed as in the previous case, this time choosing $k, h\in \SL(2,\ratk)$ for which $kx \in \{x_0,x_1\}$ and for which ${}^hT = S$.
\end{proof}

\subsection{Main theorem for representations of positive depth}
Before stating the next theorem, we require a short lemma about filtrations of tori.

\begin{lemma}\label{L:toriintersection}
Let $T$ be a maximal torus of $G=\SL(2,\ratk)$, let $x\in \buil(G)$.  Let $\dot{X}\in \g$ be nilpotent and denote its centralizer in $G_x$ by $U_x$.  Then for any $\ell\in \Z_+$, we have
$$
T \cap U_xG_{x,\ell} \subseteq ZT_{\ell}.
$$
\end{lemma}

\begin{proof}
As in the proof of Lemma~\ref{L:CGamma}, we may choose an apartment $\apart(G,S)\subset\buil(G)$ containing $x$ with root system $\Phi = \{\pm \alpha\}$ such that $\dot{X} \in \g_\alpha$; then $U_x =  (G_x\cap U_{\alpha})Z$.  Given $t\in T\cap U_xG_{x,\ell}$, write $t=uzg$ with $z\in Z$, $u\in G_x\cap U_{\alpha}$ and $g\in G_{x,\ell}$.  Then $z^{-1}t \in (G_{x,\ell}\cap U_{-\alpha})S_\ell(G_x\cap U_{\alpha})$, so its trace is  that of an element of $S_\ell$.  Since for any torus $T$ of $G$ and any $\ell>0$, we have $t'\in T_\ell$ if and only if the trace of $t'$ lies in $2+\PP^{2\ell}$, we conclude that $t\in T_\ell$.
\end{proof}

\begin{lemma}\label{L:toriintersection2}
Let $T=\Cent_{G}(\Gamma)$, where $\Gamma$ is a semisimple element of depth $-r$.  Suppose that at $x\in\buil(G)$ we have $d_{x}(\Gamma)=-d<-r$.  
Then $T\cap G_x = ZT_{d-r}$, so that $T\cap G_{x,\ell} = T_{d-r+\ell}$ for any $\ell \geq 0$.
\end{lemma}

\begin{proof}
Let $\LieT$ be the Lie algebra of $T$.  The hypotheses imply that $d_x(\varpi^k\Gamma)=d(\varpi^k\Gamma)-(d-r)$ for any $k\in \Z$.  Thus since $\LieT$ is one-dimensional, for any element $\dot{X}\in \LieT_\ell\setminus \LieT_{\ell+}$ 
we have $d_x(\dot{X}) = \ell-(d-r)$, yielding $\LieT\cap \g_{x,\ell} = \LieT_{d-r+\ell}$. Passage to the group yields the desired result, where at depth zero, we observe that $Z\subset T\cap G_x$ for all $x$. 
\end{proof}

\begin{theorem}\label{T:posdepth2}
Let $\pi = \pi(T,\chi)$ be an irreducible representation of $G=\SL(2,F)$ of depth $r>0$ and let $\Gamma = \Gamma_\pi \in \g^*$. 
Then for each maximal compact $G_x$, there is an integer $n_{x}(\pi)$ such that in the Grothendieck group of representations we have
\begin{equation}\label{E:posdepththm}
\Res_{G_{x,r+}}\pi \cong n_{x}(\pi)\triv + \sum_{\cO \in \nil(\Gamma)} \Res_{G_{x,r+}}\tau_{x}(\orb).
\end{equation}
That is, up to some copies of the trivial representation, the representation $\pi$ is locally completely determined by the nilpotent support of $\Gamma$.
\end{theorem}

\begin{proof}
The restriction to $G_{x,r+}$ will be trivial on any $G_x$-representations of depth less than or equal to $r$, so our first step is to match components of depth $d>r$ in $\Res_{G_x}\pi$ and in $\sum_{\orb \in \nil(\Gamma)}\tau_x(\orb)$.  
Note that the restriction to $G_{x,r+}$ is independent of the choice of central character $\zeta$ so it is omitted from the notation.

 Theorem~\ref{T:posdepth} gives the decomposition \eqref{E:basicdecomp} of the left side: the components of depth greater than $r$ are parametrized by the degenerate $G_x$-orbits of $\Gamma$ at $x$.
 From \eqref{E:deftau} we infer the decomposition of the right side:  the components are parametrized by nilpotent $G_x$-orbits in $\orb\setminus \g_{x,0}^*$ for each $\orb \in \nil(\Gamma)$.  

By Proposition~\ref{P:asymptotic_orbit_invariance}, 
each degenerate coset  $\xi = {}^g\Gamma+\g^*_{x,-d+}$, where $d=-d_x({}^g\Gamma)$,  is represented by a nilpotent element $X\in \cO({}^g\Gamma,x)$ such that moreover ${}^g\Gamma-X\in \g^*_{x,-r}$.  The $G_x$-orbit of $\xi$ determines the  $G_x$-orbit of $X$ and by definition $G\cdot X\in \nil(\Gamma)$.  
Thus for each $d>r$ there is a one-to-one correspondence between the $G_x$-orbits in $G\cdot \Gamma$ whose depth at $x$ is $-d$ and the $G_x$-orbits in $\nil(\Gamma)$ whose depth at $x$ is $-d$.  To complete the proof we need to show the corresponding representations are isomorphic upon restriction to $G_{x,r+}$.

Let $\zeta = \chi|_Z$ be the central character of $\pi$.
If $r<d\leq 2\lfloor r \rfloor + 1$ then applying Lemma~\ref{L:Shalikaequiv} to the pair $\Gamma_1 = {}^g\Gamma$ and $\Gamma_2=X$, with $s=r$ gives $\Res_{G_{x,r+}}\Sh_x({}^g\Gamma,{}^g\chi) \cong \Res_{G_{x,r+}}\Sh_x(X,\zeta)$ as required.

If $d>2r$,  then we have a stronger result.  Lemma~\ref{L:toriintersection2} implies that $\Cent_{G_x}({}^g\Gamma) = Z\;{}^gT_{d-r}\subseteq Z\;{}^gT_{r+}$, and thus ${}^g\chi$ is given on this subgroup by the central character $\zeta$. Since the chambers attached to ${}^g\Gamma$ and to $X$ by Lemma~\ref{L:CGamma} coincide, we have $J:=J_{x,{}^g\Gamma}=J_{x,X}$.
Since ${}^g\Gamma-X\in \g^*_{x,-r}\subset \g^*_{x,-d/2+}$, we have $\eta_{{}^g\Gamma}=\eta_X$ as characters of $J$.  
Moreover, since ${}^g\Gamma \in X+\g^*_{x,-r}$, we have $C_{G_x}({}^g\Gamma) \subseteq C_{G_x}(X)G_{x,d-r} \subseteq C_{G_x}(X)J$.  Therefore  $\eta({}^g\Gamma,{}^g\chi)=\eta(X,\zeta)$ as characters of this common group so that $\Sh_x({}^g\Gamma,{}^g\chi) = \Sh_x(X,\zeta)$ as representations of $G_x$.
\end{proof}

In the course of the proof we established that the components of depth $d>2r$ occurring in $\Res_{G_x}\pi$ coincide as representations of $G_x$, not just as representations of $G_{x,r+}$, with the components of depth $d>2r$ in $\sum_{\orb \in \nil(\Gamma)} \tau_{x}(\orb,\zeta)$, where $\zeta$ is the central character of $\pi$. This  was proven case by case in  \cite[Rem 7.5]{Nevins2005} and \cite[Prop 7.6]{Nevins2013}.

\begin{remark}
Comparing \eqref{E:posdepththm} with the known values of $\WF(\pi)$ from \cite[Tables 1--4]{DeBackerSally2000}  reveals that $\nil(\Gamma)=\WF(\pi)$ in all cases (\emph{cf} Theorem~\ref{T:nil=wf}).  Moreover, with the standard normalization chosen in \cite[I.8]{MoeglinWaldspurger1987}, the coefficients of the leading terms of the local character expansion agree with those of \eqref{E:posdepththm}; namely $c_\orb(\pi)=1$ for all $\orb\in \WF(\pi)$.     Thus Theorem~\ref{T:posdepth2} is a representation-theoretic analogue of the analytic local character expansion.

On the other hand, the constant term $n_{x}(\pi)$ of the decomposition \eqref{E:posdepththm} does not (and could not) agree with the constant term $c_{0}(\pi)$ of the local character expansion.  For one, $n_x(\pi)\in \Z$ whereas $c_0(\pi)$ may be half-integral; see Table~\ref{Table:c0}.  For another, $n_x(\pi)$ depends on the dimension of $\pi^{G_{x,r+}}$, which may vary based on the $G$-conjugacy class of the vertex $x\in \buil(G)$.  
\end{remark}

Let us compute the constant terms $n_x(\pi)$ explicitly.

\begin{proposition}\label{P:npi}
Let $\pi=\pi(T,\chi)$ be an irreducible representation of depth $r>0$ as in Theorem~\ref{T:posdepth}.  Then for each vertex $x\in \buil(G)$, the dimension of the subspace of $G_{x,r+}$-fixed vectors, as well as the value of the coefficient $n_{x}(\pi)$ appearing in \eqref{E:posdepththm} are as given in Table~\ref{T:c0x}.
\end{proposition}

\begin{table}[ht]
\begin{tabular}{|l|cccc|}
\hline Type of torus $T$ & split & unramified, $x_T\sim x$ &unramified, $x_T\not\sim x$ & ramified\\
depth $r$ & $r\in \Z_{>0}$ & $r\in \Z_{>0}$ & $r\in \Z_{>0}$ & $r\in \frac12 + \Z_{\geq 0}$\\ 
\hline $\dim(\pi(T,\chi)^{G_{x,r+}})$  &$(q+1)q^r$&$(q-1)q^r$&$0$&$0$\\
\hline $n_{x}(\pi)$ & $q+1$ & $q-q^r$ if $r$ is even & $1-q^r$ if $r$ is even & $\displaystyle (1-q^{r-1/2})(q+1)/2$\\
&&$1-q^r$ if $r$ is odd&$q-q^r$ if $r$ is odd&\\
\hline
\end{tabular}
\caption{The values of $n_{x}(\pi)$ appearing in \eqref{E:posdepththm} for each irreducible representation of $\SL(2,F)$ of depth $r>0$.}
\label{T:c0x}
\end{table}

\begin{proof}
Let $\pi=\pi(T,\chi)$ have depth $r>0$.  From Theorem~\ref{T:posdepth} we have the equality
\begin{equation}\label{E:valgxr}
n_{x}(\pi)=\dim(\pi^{G_{x,r+}}) - \sum_{\cO \in \nil(\Gamma)}\dim(\tau_x(\cO))^{G_{x,r+}}.
\end{equation}
Let us first compute $\dim(\pi^{G_{x,r+}})$ in each case.  If $\pi$ is a principal series representation and $B$ is a Borel subgroup containing $T$, then  $\pi^{G_{x,r+}} = \Ind_{(B\cap G_x)G_{x,r+}}^{G_x}\chi$, whence $\dim(\pi^{G_{x,r+}}) = \vert G_x/(B\cap G_x)G_{x,r+}\vert = (q+1)q^r$.  

If $\pi=\pi(T,\chi)=\cind_{TG_{x_T,r/2}}^G\kappa$ is an unramified supercuspidal representation such that some $G$-conjugate of $T$ is contained in $G_x$, then (replacing $T$ and $\pi$ by this conjugate) we have $x_T=x$ and  $\pi^{G_{x,r+}} = \Ind_{TG_{x,r/2}}^{G_x} \kappa(\chi)$.  It follows from a calculation in \cite[Prop 4.8]{Nevins2013} that independently of the parity of $r$ we have $\dim(\pi^{G_{x,r+}})=(q-1)q^r$.  On the other hand, if $T$ is not conjugate to a torus contained in $G_x$, then $\Res_{G_x}\pi(T,\chi)$ has no components of depth $r$ and $\dim(\pi^{G_{x,r+}})=0$.
Similarly, if $\pi =\pi(T,\chi)$ is a ramified supercuspidal representation, then its depth is half-integral, whence for a vertex $x$ we have $G_{x,r}=G_{x,r+}$, and thus by definition of depth $\pi^{G_{x,r+}}=\{0\}$.

On the other hand, the space of $G_{x,r+}$-fixed vectors of $\tau_x(\cO)$ is exactly the sum of its irreducible components of depth $d\leq r$.  These have total dimension $\frac12q^{d-1}(q^2-1)$ and  correspond to the $G_x$-orbits of $\cO$ whose depths $-d$ at $x$ satisfy $-r\leq -d\leq -1$.  Thus
if the parity depth of $\cO$ at $x$ is even
then
$$
\dim(\tau_x(\orb)^{G_{x,r+}}) = 
\sum_{e=1}^{\lfloor r/2 \rfloor} \frac12q^{2e-1}(q^2-1)
= \frac12q(q^{2\lfloor r/2\rfloor}-1).
$$
whereas if it is odd we have
$$
\dim(\tau_{x}(\orb)^{G_{x,r+}}) = 
\sum_{e=1}^{\lceil r/2 \rceil} \frac12q^{2e-2}(q^2-1) 
= \frac12(q^{2\lceil r/2 \rceil}-1).
$$
Note that $\dim(\tau_{x}(\orb)\oplus \tau_{x}(\orb'))^{G_{x,r+}}= \frac12(q+1)(q^{\lfloor r\rfloor} -1)$ when the $G_x$-orbits in $\cO$ and $\cO'$ have opposite parity depths at $x$.

Consequently, we can compute $n_{x}(\pi)$ using \eqref{E:valgxr} and the explicit determination of $\nil(\Gamma)$  in Lemma~\ref{L:NilGamma}, as follows.

If $\pi$ is a principal series representation, then $T$ is a torus and $\Gamma$ is split.  Thus all principal nilpotent orbits occur, yielding 
$$
n_{x}(\pi) = (q+1)q^r - 2\left(\frac12(q+1)(q^r-1)\right) = q+1.
$$
If $\pi$ is a supercuspidal representation corresponding to a ramified torus, then $\nil(\Gamma)$ consists of two nonzero orbits which will be of opposite parity at any vertex $x$, and $\lfloor r \rfloor = r-\frac12$.  Thus we simply have $n_{x}(\pi)=0-\frac12(q+1)(q^{r-1/2}-1)$.

On the other hand, if $T$ is unramified, then $\nil(\Gamma)$ consists of two nonzero orbits and at any vertex $x$, the parity of the depths of elements of these orbits is that of $d_{x}(\Gamma)$.  Thus if $x$ is $G$-conjugate to $x_T$, the orbits that occur have the same parity as $-r=d_{x_T}(\Gamma)$, so
$$
n_{x}(\pi) = \begin{cases}
(q-1)q^r - 2\left(\frac12q(q^{r}-1)\right) = q-q^r & \text{if $r$ is even;}\\
(q-1)q^r - 2\left(\frac12(q^{r+1}-1)\right) = 1-q^r & \text{if $r$ is odd.}
\end{cases}
$$
Finally, if $x$ is not $G$-conjugate to $x_T$, then $d_x(\Gamma)$ and $d_{x_T}(\Gamma)$ have opposite parity.  Thus we conclude
$$
n_{x}(\pi) = \begin{cases}
0 - 2\left(\frac12(q^{r}-1)\right) = 1-q^r & \text{if $r$ is even;}\\
0 - 2\left(\frac12q(q^{r-1}-1)\right) = q-q^r & \text{if $r$ is odd.}
\end{cases}
$$
\end{proof}

\section{The case of depth-zero representations of \texorpdfstring{$\SL(2,\ratk)$}{SL2F}}\label{S:depthzero}

To establish our result for depth-zero representations of $\SL(2,\ratk)$, we apply a result by Barbash and Moy relating the wave front set of $\pi$ to that of $\pi^{G_{x,0+}}$, viewed as a representation of $\SL(2,\resk)\cong G_{x,0}/G_{x,0+}$  (Proposition~\ref{P:BM}). We begin by recalling the representation theory of $\SL(2,\resk)$ and then the classification of depth-zero representations of $\SL(2,\ratk)$.

\subsection{Representations of \texorpdfstring{$\SL(2,\resk)$}{SL(2,k)}}
This theory is well-known and is beautifully recapped in  \cite[\S15]{DigneMichel1991}.  Let $\sG = \SL(2,\resk)$, $\sT$ a maximal torus of $\sG$ and $\overline{\chi}$ a character of $\sT$ (which is assumed to be  nontrivial if $\sT$ is anisotropic).  The irreducible representations of $\sG$ are parametrized by these pairs $(\sT,\overline{\chi})$ as follows.

If $\sT$ is split and $\overline{\chi}^2\neq \triv$ then $\sigma(\sT,\overline{\chi})$ is an irreducible principal series representation; if $\sT$ is anisotropic and $\overline{\chi}^2\neq \triv$ then $\sigma(\sT,\overline{\chi})$ is a (Deligne--Lusztig) cuspidal representation.  If $\sT$ is split and $\overline{\chi}=\triv$ then $\sigma(\sT,\overline{\chi}) = \triv \oplus \overline{\St}$ where $\overline{\St}$ denotes the Steinberg representation of $\sG$.  

For either $\sT$, when $\overline{\chi}$ is a strictly quadratic character, we obtain two irreducible representations $\sigma^u(\sT,\overline{\chi})$ for $u\in \{1,\ep\}$ (as the components of the restriction $\sigma(\sT,\overline{\chi})$ to $\SL(2,\resk)$ of a corresponding (irreducible) representation of $\GL(2,\resk)$).
They are distinguished by the theory of Gel'fand--Graev representations, as follows.

Let $X\in \g(\resk)^*\setminus\{0\}$ be nilpotent, and identify this with a nilpotent element $\dot{X}\in \g(\resk)$.  Complete this to an $\mathfrak{sl}(2,\resk)$ triple $\{\dot{Y},\dot{H},\dot{X}\}$ and let $\LieU(\resk) = \resk\dot{Y}$.  Then $X$ defines a character of $\sU=\exp(\LieU(\resk))$ by
$
\psi_X(\exp(W)) = \psi(X(W))$ for all $W\in \LieU(\resk)$.
Then the representation of $\SL(2,\resk)$ given by 
\begin{equation}\label{E:GGrep}
\gamma_{\cO} = \Ind_{\sU}^\sG \psi_X
\end{equation}
depends (up to equivalence) only on the nonzero orbit $\cO=\sG\cdot X$, and is called the Gel'fand--Graev representation of $\sG$ associated to $\cO$.

Contrary to convention, we parametrize our nonzero nilpotent orbits by upper triangular matrices $\dot{X}_u \in \g(\resk)$ as in \eqref{E:nilpotentorbits}, where $u \in \resk^\times/(\resk^\times)^2 \sim \{1,\ep\}$.  With respect to this choice, one can compute directly that the character $[\gamma_{\cO_u}]$ of the Gel'fand--Graev representation associated to $\cO=\sG\cdot X_u$ is given by
$$
[\gamma_{\cO_u}](g) = \begin{cases}
q^2-1  & \text{if $g=I$};\\
2\sigma_{-us} = 2\sum_{t \in (\resk^\times)^2}\psi(-ust) & \text{if $g \sim \smat{1&s\\0&1}$};\\
0 & \text{otherwise}.
\end{cases}
$$
By \cite[Thm 14.30]{DigneMichel1991}, the decomposition into irreducible subrepresentations of $\gamma_{\orb_u}$ is mul\-ti\-pli\-city-free.  Using character tables it is straightforward to see that it contains
 all irreducible principal series representations, all Deligne--Lusztig cuspidal representations, the Steinberg representation, and 
exactly one  from each pair of representations arising from quadratic characters.  Our parametrization is therefore as follows: 
for $u\in \resk^\times/(\resk^\times)^2$, and a quadratic character $\overline{\chi}$ of $\sT$, let $\sigma^{u}(\sT,\overline{\chi})$ denote the component of $\sigma(\sT,\overline{\chi})$ occuring in $\gamma_{\cO_u}$.   In the notation of \cite[\S 15]{DigneMichel1991}, the characters of these representations are labeled $\overline{\chi}^\pm$ where $[\sigma^{-1}(\sT,\overline{\chi})] = \overline{\chi}^+$ and $[\sigma^{-\ep}(\sT,\overline{\chi})] = \overline{\chi}^-$.

\subsection{Depth-zero representations of \texorpdfstring{$\SL(2,\ratk)$}{SL(2,F)}}\label{SS:SL2F}  Now let $G=\SL(2,\ratk)$ and let $\chi$ be a depth-zero character of a maximal split or unramified torus.  Assume $\chi$ is nontrivial if the torus is nonsplit.  There are two nonconjugate choices $T_1,T_2$ for an unramified anisotropic torus.  If $x_0$ and $x_1$ are the vertices of the standard alcove, as before, then we can choose representatives $T^i$ of the conjugacy classes of maximal tori such that $T^i\subset G_{x_i}$, for $i\in \{0,1\}$.  Then $\sT_i = T^i_0/T^i_{0+}$ is  a maximal anisotropic torus of $G_{x_i,0}/G_{x_i,0+}=:\sG_i \cong \SL(2,\resk)$.  Let $T$ denote the split torus corresponding to the standard apartment and set $\sT=T_0/T_{0+}$, which is a maximal split torus of both $\sG_1$ and $\sG_2$.   In each case, the character $\chi$ factors to a character $\overline{\chi}$ of the quotient.

In the nonsplit case, for each $i\in \{0,1\}$ inflate the representation $\sigma(\sT_i,\overline{\chi})$ of $\sG_i$ to a representation of $G_{x_i}$ and define $\pi(T^i,\chi) = \cind_{G_{x_i}}^G \sigma(\sT_i,\overline{\chi})$ when $\chi^2\neq \triv$.  When $\chi^2=\triv$ set $\pi^u(T^i,\chi) = \cind_{G_{x_i}}^G \sigma^u(\sT_i,\overline{\chi})$ for $u\in \{1,\ep\}$.  These representations are supercuspidal and irreducible; the latter four were called the special representations.   For $\eta = \smat{1&0\\0&\varpi} \in \GL(2,\ratk)$ we have ${}^\eta \pi^*(T^0,\chi)=\pi^*(T^1,{}^\eta\chi)$, where $*$ indicates that this applies both to the special and nonspecial representations.  
It follows from \cite[Thm 5.3]{Nevins2013} that for any vertex $x=gx_i$, with $g\in G$, the depth-zero component $\pi^*(T^i,\chi)^{G_{x,0+}}$ is the inflation to $G_x$ of ${}^g\sigma^*(\sT_i,\overline{\chi})$, but $\pi^*(T^i,\chi)^{G_{x,0+}}=\{0\}$ if $x$ is not in the $G$-orbit of $x_i$.  

If $T$ is split, contained in a Borel subgroup $B$, then $\pi(T,\chi) = \Ind_B^G\chi$ is again in the principal series.    It is immediate to see that for any vertex $x$, $\pi(T,\chi)^{G_{x,0+}} \cong \sigma(\sT,\overline{\chi})$ under the isomorphism $G_{x,0}/G_{x,0+}\cong \SL(2,\resk)$. In fact, $\pi(T,\chi)$ is irreducible if and only if  $\sigma(\sT,\overline{\chi})$  is; its factors in the remaining cases are as follows.

When $\chi \in \{\nu,\nu^{-1}\}$, its Jordan-H\"older factors are the trivial representation and the Steinberg representation $\St$, and we have $\St^{G_{x,0+}} = \overline{\St}$ and $\triv^{G_{x,0+}} = \triv$.

When $\chi$ is quadratic, it is the sign character $\sgn_\tau$ corresponding to the extension $\extk=\ratk[\sqrt{\tau}]$. As described in \cite[\S8]{Nevins2005}, there is in this case a realization of $\pi(T,\chi)$ on the space $L^2(\ratk^\times)$  such that its two irreducible summands $H^\tau_i$, with $i\in \{+,-\}$, 
are the functions supported $N_i^\tau=\{u\in \ratk^\times\mid i\sgn_\tau(u)=1\}$, respectively.    
Note that the cases $-1\in (\ratk^\times)^2$ and $-1\notin (\ratk^\times)^2$ thus need to be considered separately while computing the components of depth zero, but in fact the result may be stated uniformly, as follows. 

\begin{proposition}\label{P:pigx0+}
For each $\tau\in \{\ep,-\varpi,-\ep\varpi\}$ and $i\in \{0,1\}$, the $\sG_i$-representations $(H_\pm^\tau)^{G_{x_i,0+}}$ are  irreducible and their isomorphism classes are given in Table~\ref{Table:H+}.
\end{proposition}

\begin{table}[ht]
\begin{tabular}{|c||cc|cc|cc|}
\hline $\pi$ & $H^\ep_{+}$ & $H^\ep_{-}$ & $H^{-\varpi}_{+}$ & 
$H^{-\varpi}_{-}$ & $H^{-\ep\varpi}_{+}$ & $H^{-\ep\varpi}_{-}$\\
\hline 
$\pi^{G_{x_0,0+}}$ & $\overline{\St}$ & $\triv$ & $\sigma^{1}(\sT,\sgn)$ & $\sigma^{\ep}(\sT,\sgn)$ & $\sigma^{1}(\sT,\sgn)$ & $\sigma^{\ep}(\sT,\sgn)$\\
$\pi^{G_{x_1,0+}}$ & $\triv$ & $\overline{\St}$ & $\sigma^{1}(\sT,\sgn)$ & $\sigma^{\ep}(\sT,\sgn)$ & $\sigma^{\ep}(\sT,\sgn)$ & $\sigma^{1}(\sT,\sgn)$\\
\hline
\end{tabular}
\caption{The isomorphism classes of the depth-zero representations of $G_{x_i}$ occuring in the restriction to $G_{x_i}$ of the decomposable principal series.}\label{Table:H+}
\end{table}

\begin{proof}
Without loss of generality we may assume $x_0,x_1 \in \apart(G,T)$.  Since $\overline{\sgn_\ep} = \triv$ and $\overline{\sgn_{-\varpi}}=\overline{\sgn_{-\ep\varpi}}=\sgn$, the unique quadratic character of $\sT$, we immediately have the relation
$$
(H^\tau_{+})^{G_{x,0+}} \oplus (H^\tau_{-})^{G_{x,0+}} \cong 
\sigma^1(\sT,\overline{\chi})\oplus \sigma^\ep(\sT,\overline{\chi}),
$$
for any vertex $x\in \apart(G,T)$.

The restriction to $G_{x_0,0+}$ of these components was determined via character computations in \cite[Thm 9.1]{Nevins2005}, where (in the notation of that paper, of \cite[\S15]{DigneMichel1991}, and ours, respectively), $\Xi^+_{\sgn}=\chi^-_{\alpha_0}=[\sigma^{-\ep}(\sT,\sgn)]$ and $\Xi^-_{\sgn}=\chi^+_{\alpha_0}=[\sigma^{-1}(\sT,\sgn)]$.  The negative signs used in our parametrizations simplify the statement of the result, yielding the first row of the table.  

For the restriction to $G_{x_1,0+}$, the proof of \cite[Cor 9.3]{Nevins2005} showed that twisting $\pi(T,\sgn_\tau)$ by $\omega=\smat{0&1\\\varpi&0} \in \GL(2,\ratk)$ preserves $H_{\pm}^\tau$ when $\sgn_\tau(-\varpi)=1$ and interchanges them otherwise.  Applying this twist to $\pi^{G_{x_0,0+}}$ yields $\pi^{G_{x_1,0+}}$ and sends a representation $\sigma$ of $G_{x_0}$ to the representation ${}^\omega \sigma$ of $G_{x_1}$.  Twisting by $\omega$ sends the inflation of the  representation $\sigma^u(\sT,\overline{\chi})$ of $G_{x_0}$ to the inflation of the  representation $\sigma^{-u}(\sT,\overline{\chi})$ of $G_{x_1}$, since it takes $\orb_u$ to $\orb_{-u\varpi}$, which in turn determines the choice of Gel'fand--Graev representation $\gamma_{\orb_u}$.\footnote{This calculation was neglected in the proof of \cite[Cor 9.3]{Nevins2005}, yielding an incorrect statement for the depth-zero components.} A careful accounting of signs yields the second row of the table. 
\end{proof}

For convenience, we list the isomorphism class of $\pi^{G_{x,0+}}$ for the remaining irreducible representations $\pi$ in Table~\ref{Table:depthzerocomponents}. 

\begin{table}[ht]
\begin{tabular}{|c|c|cc|}
\hline $T$ & $\pi$ & $\pi^{G_{x,0+}}$& \\
\hline \hline
$T$ split & $\pi=\pi(T,\chi)$ & $\sigma(\sT, \overline{\chi})$ & \\ \hline
&  $\pi=\St$ &  ${}^g\overline{\St}$ &   \\ \hline \hline
$T^i$ unramified &  $\pi = \pi(T^i,\chi)$ & $\sigma(\sT_i, \overline{\chi})$ & if $x \sim x_i$\\
$i\in \{0,1\}$ & & $\{0\}$ & if $x\not\sim x_i$ \\ \hline
\end{tabular}
\caption{The depth-zero representations of $G_{x}$ occuring in the restriction to $G_x$ of the irreducible principal series, Steinberg and supercuspidal representations, for any vertex $x\in \buil(G)$.}\label{Table:depthzerocomponents}
\end{table}

\subsection{Wave front sets}
The wave front set is determined with the following result that is based on \cite[Thm 4.5]{BarbaschMoy1997}.

\begin{proposition}\label{P:BM}
Let $\pi$ be an irreducible representation of depth zero of $\SL(2,\ratk)$.  Suppose $\mathrm{char}(F)=0$ and $p>3e+1$, where $e$ is the absolute ramification index of $\ratk$ over $\mathbb{Q}_p$.  Then we have
$$
\WF(\pi) = \{ \cO \in \setorbit \mid \exists x \text{ a vertex of }\buil(G), \sigma \in \pi^{G_{x,0+}}  \text{such that $\sigma$ occurs in $\gamma_\cO$}\},
$$
where $\gamma_\cO$ is the Gel'fand--Graev representation \eqref{E:GGrep} of $\sG_x \cong \SL(2,\resk)$.
\end{proposition}

\begin{proof}
The hypotheses imply that $\exp$ converges on $\g_{0+}$ and that the local character expansion holds.  In \cite{BarbaschMoy1997}, Barbasch and Moy used (generalized) Gel'fand--Graev characters as test functions to determine the  wave front set of $\pi$.  For each nilpotent orbit $\cO$ that is represented by a depth-zero coset at some $x$, let $[\gamma_\cO]$ denote the lift to $G_{x,0}$ of the character of the corresponding Gel'fand--Graev representation of $\sG_x = G_{x,0}/G_{x,0+}$, viewed as a function on $G$.  It is supported on the subset $G_{0+}\cap G_{x,0}$ of topologically unipotent elements.  Let $f_{x,\cO}$ be the function on $\g$, with support in $\g_{0+}\cap\g_{x,0}$, that is given by $f_{x,\cO} = [\gamma_\cO]\circ \exp$.
 Then they show that $\widehat{\mu_{\cO'}}(f_{x,\cO})=0$ if $\cO\not\subset \overline{\cO'}$ and is nonzero if $\cO=\cO'$.  Thus $\Theta_\pi([\gamma_\cO]) =0$ for all $\cO$ that do not meet the wave front set of $\pi$ and is nonzero when $\cO\in \WF(\pi)$.  

For any irreducible representation $\sigma$ of $\sG_x$, let $m(\sigma,\pi)$ denote the multiplicity of (the inflation of) $\sigma$ in $\pi^{G_{x,0+}}$ and $m(\sigma,\gamma_\cO)$ the multiplicity of $\sigma$ in $\gamma_\cO$.  Then \cite[Thm 4.5(4)]{BarbaschMoy1997} becomes
$$
\Theta_\pi([\gamma_{\cO}]) = \sum_{\sigma}m(\sigma,\pi)m(\sigma,\gamma_\cO),
$$
whence our result for the case of $\SL(2,\ratk)$.
\end{proof}

\begin{corollary}
Under the hypothesis of Proposition~\ref{P:BM}, the wave front sets corresponding to the depth-zero representations of $\SL(2,\ratk)$ are as given in Table~\ref{Table:WF}.
\end{corollary}

\begin{table}
\begin{tabular}{||p{1in}|p{1in}||p{1in}|p{1in}||} 
\hline \multicolumn{3}{||c|}{Representation $\pi$} & $\WF(\pi)$\\
\hline
\multicolumn{3}{||c|}{$\pi(T,\chi)$ 	irreducible principal series}	&	$\{1,\ep,\varpi,\ep\varpi\}$\\ 
\multicolumn{3}{||c|}{$\pi(T_0,\chi)$ 	irreducible supercuspidal 	}	&$\{1,\ep\}$\\ 
\multicolumn{3}{||c|}{$\pi(T_1,\chi)$ 	irreducible supercuspidal 	}	&$\{\varpi,\ep\varpi\}$\\ \hline \hline
$\pi$ & $\WF(\pi)$ & $\pi$ & $\WF(\pi)$\\ \hline
$\triv$	&$\{0\}$	&$\St$&	$\{1,\ep,\varpi,\ep\varpi\}$\\
$H^\ep_+$	&$\{1,\ep\}$ &	$H^\ep_{-}$ 	&$\{\varpi,\ep\varpi\}$\\
$H^{-\varpi}_+$&	$\{1,\varpi\}$ &	$H^{-\varpi}_{-}$ &	$\{\ep,\ep\varpi\}$\\
$H^{-\ep\varpi}_+$	&$\{1,\ep\varpi\}$ &	$H^{-\ep\varpi}_{-}$ &	$\{\ep,\varpi\}$\\
$\pi^1(T_0,\sgn)$&	$\{1\}$	&$\pi^\ep(T_0,\sgn)$	&$\{\ep\}$\\ 
$\pi^1(T_1,\sgn)$&	$\{\varpi\}$	&$\pi^\ep(T_1,\sgn)$	&$\{\ep\varpi\}$\\
\hline
\end{tabular}
\caption{For each irreducible depth-zero representation of $\SL(2,\ratk)$, we list under the heading $\WF(\pi)$ the set of elements $u \in \{0,1,\ep,\varpi,\ep\varpi\}$ such that $\orb_u\in \WF(\pi)$.}
\label{Table:WF}
\end{table}

\begin{proof}
For $u\in \{1,\ep\}$ and $j \in \{0,1\}$, the nilpotent orbit $\orb_{u\varpi^j}$ is represented by a depth-zero coset at $x_i$  if and only if $i=j$, and in this case it corresponds to the nilpotent orbit in the quotient $\g_{x_i,0}/\g_{x_i,0+}\cong \mathfrak{sl}(2,\resk)$ under $\sG_i\cong \SL(2,\resk)$ that we denoted $\orb_u$.  Therefore the Gel'fand--Graev representations $\gamma_\orb$ referred to in  Proposition~\ref{P:BM} are $\gamma_{\orb_1}$ and $\gamma_{\orb_\ep}$ for $x=x_0$, and $\gamma_{\orb_{\varpi}}$ and $\gamma_{\orb_{\ep\varpi}}$ for $x=x_1$.
By conjugacy, these two vertices suffice.  The decomposition of  $\pi^{G_{x,0+}}$ for $x\in \{x_0,x_1\}$ was given in  Tables~\ref{Table:H+} and \ref{Table:depthzerocomponents} for all irreducible depth-zero representations $\pi$, and matching these with the decomposition of the Gel'fand--Graev representations of the corresponding groups $\sG_i$ yields Table~\ref{Table:WF}.
\end{proof}

In light of Proposition~\ref{P:BM}, we may define $\WF(\pi)$ by Table~\ref{Table:WF}, even over fields where Proposition~\ref{P:BM} does not apply.  Theorem~\ref{T:zerodepth2} below expresses that, just as in the positive-depth case, this is consistent (for all fields with residual characteristic different from $2$).

\begin{corollary}
For each depth-zero irreducible representation $\pi$ of $\SL(2,\ratk)$ there exists an element $\Gamma\in \g_{x,0}^*$, for some $x\in \buil(G)$, such that $\WF(\pi)=\nil(\Gamma)$.
\end{corollary}

Existence follows immediately from Table~\ref{Table:WF} and Lemma~\ref{L:NilGamma}, though the elements $\Gamma$ for which $\WF(\pi)=\nil(\Gamma)$  do not correspond to  minimal $K$-types for $\pi$ (as these latter are not realized by elements on the Lie algebra).  However, on an \emph{ad-hoc} basis, we can make this association of $\pi$ with $\Gamma$ more explicit, as follows.  

For $T$ unramified or split, and $x\in \buil(T)\subset \buil(G)$, we can in the same spirit attach to any \emph{regular} $\pi(T,\chi)$  (in the sense of Kaletha \cite[Prop 3.4.27]{Kaletha2019})  any element $\Gamma\in \g_{x,0}^*$ whose centralizer in $G$ is $T$. The same holds for $\pi=\St$, whereas we associate $\Gamma=0$ to $\triv$.

When $\pi^u(T^i,\sgn)$ is a special representation (for some $u\in \{1,\ep\}$ and $i\in \{0,1\}$) then it is a supercuspidal unipotent representation and $\Gamma$ is chosen to be a nilpotent element in the lift to $\g_{x_i,0}^*$ of the nilpotent orbit corresponding to $\sigma^u(\sT_i,\sgn)$.

When $\pi \in \{H^\tau_{\pm} \mid \tau \in \{\ep,\varpi,\ep\varpi\}\}$,   $\Gamma$ is a choice of element of an anisotropic torus $T$ that splits over $\ratk[\sqrt{\tau}]$. However, while the orbit of $\Gamma$ satisfies $\nil(\Gamma)=\WF(\pi)$, its centralizer in this case need not correspond to $\pi$: when $-1\in (\ratk^\times)^2$ and $\tau \in \{\varpi,\ep\varpi\}$,  the centralizer may be one of two possible tori $T = \Cent_G(\Gamma)$ up to conjugacy, and neither one is expressly associated to $\pi$.

We conclude this section with our main result.

\begin{theorem}\label{T:zerodepth2}
Let $\pi$ be an irreducible representation of $G$ of depth zero with central character $\zeta$. For any vertex $x\in \buil(G)$, we have
\begin{equation}\label{E:desired}
\Res_{G_{x}}\pi \cong \pi^{G_{x,0+}}  \oplus \bigoplus_{\cO \in \WF(\pi)}\tau_x(\cO,\zeta),
\end{equation}
where $\WF(\pi)$ is as in Table~\ref{Table:WF}.
It follows that $\Res_{G_{x,0+}}\pi$ takes the form of \eqref{E:posdepththm} with constant coefficient $n_x(\pi)=\dim(\pi^{G_{x,0+}})$.  
\end{theorem}

\begin{proof}
The decomposition will follow from the main results of  \cite{Nevins2005,Nevins2013} as in the proof of Theorem~\ref{T:posdepth}.  Let $\pi$ be a  depth-zero representation of $G$ with central character $\zeta$, and let $x\in \{x_0,x_1\}$.

For irreducible depth-zero principal series, one has $\gamma_0=\gamma_1=0$ in \cite[Thm 7.4]{Nevins2005}.  Matching notation as is \eqref{NotationNevins2005}, we conclude that $\Sh_{x}(X_{u\varpi^{-d}},\zeta)$ occurs in $\Res_{G_{x}}\pi$, for $u\in \{1,\ep\}$, for each $d>0$, and that these exhaust the irreducible summands.  Therefore the summands can be regrouped as the sum of $\tau_{x}(\cO,\zeta)$, as defined in  \eqref{E:deftau},  over all regular nilpotent orbits, as required.   As the positive-depth summands of $\Res_{G_x}\pi$ are identical for all depth-zero irreducible principal series,  the case of $\pi=\St$ follows since $\Res_{G_x}\triv$ has no positive-depth components.


The results of  \cite[Thm 9.1, Thm 9.2]{Nevins2005} yield 
$$
\Res_{G_{x_0}}H^\tau_+ =
\begin{cases}
 \St \oplus \bigoplus_{d>0}(\Sh_{x_0}(\varpi^{-2d}X_1, \zeta) \oplus \Sh_{x_0}(\varpi^{-2d}X_\ep,\zeta)) & \text{if $\tau=\ep$};\\
 \sigma^1(\sT,\sgn)  \oplus \bigoplus_{d>0}\Sh_{x_0}(\varpi^{-d}X_1, \zeta) & \text{if $\tau=-\varpi$};\\
 \sigma^1(\sT,\sgn) \oplus \bigoplus_{d>0}(\Sh_{x_0}(\varpi^{-2d}X_1, \zeta) \oplus \Sh_{x_0}(\varpi^{-2d+1}X_\ep,\zeta))& \text{if $\tau=-\ep\varpi$.}
\end{cases}
$$
Regrouping the positive-depth summands yields the decomposition
$$
\Res_{G_{x_0}}H^\tau_+ =
\begin{cases}
 \St \oplus \tau_{x_0}(\orb_1,\zeta) \oplus \tau_{x_0}(\orb_\ep, \zeta)& \text{if $\tau=\ep$};\\
 \sigma^1(\sT,\sgn)  \oplus \tau_{x_0}(\orb_1,\zeta) \oplus \tau_{x_0}(\orb_{\varpi},\zeta) & \text{if $\tau=-\varpi$};\\
 \sigma^1(\sT,\sgn) \oplus \tau_{x_0}(\orb_1,\zeta) \oplus \tau_{x_0}(\orb_{\ep\varpi},\zeta) & \text{if $\tau=-\ep\varpi$,}
\end{cases}
$$
which coincides with the wave front set computed in  Table~\ref{Table:H+}.  Since the positive-depth summands of $H^\tau_+\oplus H^\tau_-$ form $\bigoplus_{\cO \in \setorbit\setminus \{0\}}\tau_{x_0}(\cO,\zeta)$, and the wave front sets of these representations are complementary, this yields the result for $\Res_{G_{x_0}}H^\tau_-$ as well.  

To determine $\Res_{G_{x_1}}\pi$ we proceed as in the proof of Proposition~\ref{P:pigx0+}.  Conjugation by $\omega$ interchanges the components of the principal series \emph{except} when: $\tau=-\varpi$ and $-1\in  (\ratk^\times)^2$; or $\tau=-\ep\varpi$ and $-1\notin  (\ratk^\times)^2$. Since the depth-zero components were computed in Proposition~\ref{P:pigx0+} and ${}^\omega\tau_{x_0}(\orb_u,\zeta) =\tau_{x_1}(\orb_{-u\varpi},\zeta)$, we deduce that 
$$
\Res_{G_{x_1}}H^\tau_- =
\begin{cases}
 \St \oplus \tau_{x_1}(\orb_\varpi,\zeta) \oplus \tau_{x_1}(\orb_{\ep\varpi}, \zeta)& \text{if $\tau=\ep$};\\
 
  \sigma^{-1}(\sT,\sgn)  \oplus \tau_{x_1}(\orb_{-\varpi},\zeta) \oplus \tau_{x_1}(\orb_{-\varpi^2},\zeta) & \text{if $\tau=-\varpi$ and $-1\notin  (\ratk^\times)^2$};\\
 \sigma^{-1}(\sT,\sgn) \oplus \tau_{x_1}(\orb_{-\varpi},\zeta) \oplus \tau_{x_1}(\orb_{-\ep\varpi^2},\zeta) 
 & \text{if $\tau=-\ep\varpi$ and $-1\in  (\ratk^\times)^2$;}\\
 \sigma^{-\ep}(\sT,\sgn)  \oplus \tau_{x_1}(\orb_{-\ep\varpi},\zeta) \oplus \tau_{x_1}(\orb_{-\ep\varpi^2},\zeta) 
 & \text{if $\tau=-\varpi$ and $-1\in  (\ratk^\times)^2$};\\
  \sigma^{-\ep}(\sT,\sgn) \oplus \tau_{x_1}(\orb_{-\ep\varpi},\zeta) \oplus \tau_{x_1}(\orb_{-\varpi^2},\zeta) 
 & \text{if $\tau=-\ep\varpi$ and $-1\notin  (\ratk^\times)^2$;}
\end{cases}
$$
Thus, in any case, the nilpotent orbits arising in $\Res_{G_{x_1}}H^\ep_-$ are $\{\orb_\varpi, \orb_{\ep\varpi}\}$; those arising in $\Res_{G_{x_1}}H^{-\varpi}_-$ are $\{\orb_\ep, \orb_{\ep\varpi}\}$; and those arising in $\Res_{G_{x_1}}H^{-\ep\varpi}_-$ are $\{\orb_{\varpi},\orb_{\ep}\}$, which again is consistent with Table~\ref{Table:WF}, as required.

Now suppose that $\pi_i = \cind_{G_i}^G\sigma$ is a supercuspidal representation.  We use \cite[Cor 5.2, Thm 5.3]{Nevins2013},  where $\eta=\smat{1&0\\0&\varpi}$, $\sigma_0^+$ corresponds to our $\sigma^{-1}(T^0,\chi)$, and $\sigma_0^-$ is our $\sigma^{-\ep}(T^0,\chi)$.  Since ${}^\eta\orb_u=\orb_{u\varpi}$, the $G_{x_1}$-representation ${}^\eta(\sigma_0^+)$ is the inflation of $\sigma^{-1}(T^1,\chi)$ to $G_{x_1}$.  We thus infer the decompositions 
$$
\Res_{G_{x_0}}\pi = 
\begin{cases}
\sigma \oplus \bigoplus_{t>0}\left(\Sh_{x_0}(-\varpi^{-2t}X_1,\zeta) \oplus \Sh_{x_0}(-\varpi^{-2t}X_\ep,\zeta)\right) & \text{if $\pi$ nonspecial and $i=0$;}\\
\bigoplus_{t>0}\left(\Sh_{x_0}(-\varpi^{-2t+1}X_1,\zeta) \oplus \Sh_{x_0}(-\varpi^{-2t+1}X_\ep,\zeta)\right) & \text{if $\pi$ nonspecial and $i=1$;}\\
\sigma \oplus \bigoplus_{t>0} \Sh_{x_0}(-\varpi^{-2t}X_u, \zeta) & \text{if $\pi = \pi^{-u}(T^0,\chi)$ is special and $i=0$;}\\
\bigoplus_{t>0} \Sh_{x_0}(-\varpi^{-2t+1}X_u, \zeta) & \text{if $\pi = \pi^{-u}(T^1,\chi)$ is special and $i=1$.}
\end{cases}
$$
Comparing with Table~\ref{Table:WF}, we conclude that \eqref{E:desired} holds for $\Res_{G_{x_0}}\pi_i$ in each case.  The result for general $x$ now follows as in the proof of   Theorem~\ref{T:posdepth}.

Finally, the values $n_x(\pi)=\dim(\pi^{G_{x,0+}})$ can be deduced from Tables~\ref{Table:H+} and \ref{Table:depthzerocomponents}:  it is $q+1$ for irreducible principal series,   $q-1$ for Deligne-Lusztig cuspidal representations, $q$ for $\overline{\St}$, $(q-1)/2$ for the special unipotent representations and  $(q+1)/2$ for the components of the reducible principal series.
\end{proof}

\section{Applications} \label{S:applications}

\subsection{The Fourier transform of a nilpotent orbital integral} \label{S:FT}

As a first application, we derive a formula for the Fourier transform of a nilpotent orbital integral in any open set of the form $\g_{x,0+}$ in terms of the trace characters of the representations $\tau_x(\cO,\zeta)$.  

\begin{proposition}\label{C:orbint}
Let $x\in \buil(G)$ be a vertex.  Let $[\tau_x(\cO)]$ denote the restriction to $G^{\reg}_{x,0+}$ of the trace character of the representation $\tau_x(\cO,\zeta)$, for either choice of central character $\zeta$.   Assume $\exp$ converges on $\g_{x,0+}$.  Then 
for each nonzero nilpotent orbit $\cO$ and $X\in \g^{\reg}_{x,0+}$ we have 
$$
\widehat{\mu_\cO}(X) =
\begin{cases} 
q/2 + [\tau_x(\cO)](\exp X) &\text{if $\cO$ has even parity depth at $x$;}\\
1/2 + [\tau_x(\cO)](\exp X) & \text{if $\cO$ has odd parity depth at $x$}.
\end{cases}
$$
As $x$ ranges over the vertices of $\buil(G)$, these expressions determine the function $\widehat{\mu_\orb}$ on $\g^{\reg}_{1/2+}$.
\end{proposition}

\begin{proof}
Let $\Theta_\pi$ denote the character of the depth-$r$ representation $\pi$.
We assume the functions $\widehat{\mu}_{\orb}$ are normalized as in \cite{MoeglinWaldspurger1987}, so that the coefficients $c_\orb$ corresponding to $\orb\in \WF(\pi)$ in the local character expansion of $\Theta_\pi\circ \exp$ are all equal to $1$.  Thus on $\g^{\reg}_{x,r+}$ we have
$$
\Theta_\pi\circ \exp = c_0(\pi) + \sum_{\cO\in \WF(\pi)}\hat{\mu}_\cO.
$$
 The constant coefficients for supercuspidal representations are given in Table~\ref{Table:c0}, following, for example, \cite[Tables 1--4]{DeBackerSally2000}. For (irreducible components of) principal series, the constant term of the local character expansion is trivial, except in the case of the trivial and Steinberg representations, which have constant terms $1$ and $-1$, respectively.

\begin{table}[ht]
\begin{tabular}{|c|c|}
\hline Representation of $\SL(2,\ratk)$ & coefficient $c_0$ of $\mu_{\{0\}}$\\
of depth $r\geq 0$ & in local character expansion\\
\hline
$\pi(T,\chi)$, $T$ unramified & $-q^r$\\
$\pi^u(T^i,\chi)$, $i\in \{0,1\}$, $u\in \{1,\ep\}$ & $-1/2$\\
$\pi(T,\chi)$, $T$ ramified & $q^{r-1/2}(q+1)/2$\\
$\St$, Steinberg representation & $-1$\\
\hline
\end{tabular}
\caption{Values of the constant term in the local character expansion of supercuspidal and Steinberg representations of $\SL(2,\ratk)$. }
\label{Table:c0}
\end{table}

Theorem~\ref{T:zerodepth2}, on the other hand, gives a formula for the character of any irreducible depth-zero representation on $G_{x,0+}$.  Matching these for the special unipotent representations $\pi = \pi^u(T,\chi)$ yields the given formula.  It is moreover direct to verify the consistency of this expression across the local character expansions of all irreducible representations, including those of positive depth (on $G_{x,r+}$  as in Theorem~\ref{T:posdepth2}).

Finally, we note that for $G=\SL(2,\ratk)$ we have
$\g_{1/2+} = G\dot (\g_{x_0,0+} \cup \g_{x_1,0+}) \subsetneq \g_{0+}$, which limits the $G$-domain on which the formulas hold. 
\end{proof}

\begin{remark}\label{restrep}
Much more explicit formulae for the functions $\widehat{\mu}_\orb$ have been computed for the group $\SL(2,\ratk)$ in \cite{Assem1994,DeBackerSally2000} among others.  They have also noted that, under the exponential map, the characters of the five representations $\{1, \pi^u(T^i,\chi)\mid u\in \{1,\ep\}, i\in \{0,1\}\}$ form another basis for the span of the functions $\widehat{\mu}_\orb$.  In fact the special representations have local character expansions of the form
\begin{equation}\label{E:specialchar}
\Theta_\pi(\exp(X))=\widehat{\mu_\orb}(X) - 1/2,
\end{equation}
for the single corresponding orbit $\orb$, and this 
 holds on the strictly larger set $\g^{\reg}_{0+}$. 
 \end{remark}
 
An advantage to Proposition~\ref{C:orbint} is the simplicity and explicitness of the construction, which uses no more than a vertex and a representative of the orbit as input.  In this, it recalls some of the original formulae for these Fourier transforms of nilpotent orbital integrals in \cite{HarishChandra1999}.

\subsection{Computing the polynomial \texorpdfstring{$\dim(\pi^{G_{x,2n}})$}{dim(pi2n)}}  This arose from a question posed to me by Marie-France Vign\'eras in 2022 and answers \cite[Question 1.1]{HenniartVigneras2023} to the negative (not unexpectedly) for this case.  If $\pi$ is an irreducible representation of $G$, the local character expansion implies that $\dim(\pi^{G_{x,2n}})$ is expressible as a polynomial in $q$, as described in \cite[\S5.1]{BarbaschMoy1997}; see also \cite[Remark 11.8]{HenniartVigneras2023}.  Here we can obtain this polynomial as a corollary of Theorem~\ref{T:posdepth2} and \ref{T:zerodepth2}, using the explicit values computed in Proposition~\ref{P:npi}.

\begin{corollary}
Let $\pi$ be an irreducible representation of $G=\SL(2,\ratk)$ of depth $r$.  Then for each integer $n>0$, we have
$$
\dim(\pi^{G_{x,2n}}) = 
\begin{cases}
q^{2n}+q^{2n-1} & \text{if $\pi$ is an irreducible principal series},\\
q^{2n-1}-q^r & \text{if $\pi$ is supercuspidal nonspecial, from a vertex $\sim x$},\\
q^{2n}-q^r & \text{if $\pi$ is supercuspidal nonspecial, from a vertex $\not\sim x$},\\
\frac12(q+1)(q^{2n-1}-q^{r-\frac12}) & \text{if $\pi$ is supercuspidal, from a nonvertex}.
\end{cases}
$$
On the other hand, if $\pi_s = H^\ep_{s}$ then 
$\dim(\pi^{G_{x,2n}}) =q^{2n-1}$ when the parity depth at $x$ of the orbits in $\WF(\pi_s)$ is even, and equals $q^{2n}$ otherwise; and if $\pi=\St$, then $\dim(\pi^{G_{x,2n}})=q^{2n}+q^{2n-1}-1$.  In all other cases, $\dim(\pi^{G_{x,2n}})$ is exactly half of that of a corresponding nonspecial representation.
\end{corollary}

\providecommand{\bysame}{\leavevmode\hbox to3em{\hrulefill}\thinspace}
\providecommand{\MR}{\relax\ifhmode\unskip\space\fi MR }
\providecommand{\MRhref}[2]{%
  \href{http://www.ams.org/mathscinet-getitem?mr=#1}{#2}
}
\providecommand{\href}[2]{#2}

\end{document}